\documentclass[12pt]{amsart}

\usepackage{amssymb,amsbsy}
\usepackage{latexsym,color,longtable}
\usepackage{verbatim,euscript}
\usepackage{fullpage,array}
\usepackage{tikz}
\usetikzlibrary{arrows,shapes,trees}
\usepackage[matrix,arrow,curve]{xy}

\numberwithin{equation}{section}
\usepackage[colorlinks=true,linkcolor=blue,urlcolor=violet,citecolor=magenta]{hyperref}

\input {cyracc.def}
\font\tencyr=wncyr10 
\def\rus{\tencyr\cyracc}

\catcode`\@=\active
\catcode`\@=11

\renewcommand{\@cite}[2]{[{{\bf #1}\if@tempswa , #2\fi}]}
\renewcommand{\@biblabel}[1]{[{\bf #1}]\hfill}
\catcode`\@=12

\newtheorem{thm}{Theorem}[section]
\newtheorem{lm}[thm]{Lemma}
\newtheorem{cl}[thm]{Corollary}
\newtheorem{prop}[thm]{Proposition}

\theoremstyle{remark}
\newtheorem{rmk}[thm]{Remark}

\theoremstyle{definition}
\newtheorem{ex}[thm]{Example}

\newtheorem*{rema}{Remark}

\newcommand {\ah}{{\mathfrak a}}

\newcommand {\g}{{\mathfrak g}}
\newcommand {\h}{{\mathfrak h}}

\newcommand {\el}{{\mathfrak l}}
\newcommand {\me}{{\mathfrak m}}

\newcommand {\q}{{\mathfrak q}}

\newcommand {\z}{{\mathfrak z}}

\newcommand {\gln}{{\mathfrak {gl}}_n}
\newcommand {\sln}{{\mathfrak {sl}}_n}
\newcommand {\slno}{{\mathfrak {sl}}_{n+1}}

\newcommand {\spn}{{\mathfrak {sp}}_{2n}}

\newcommand {\son}{{\mathfrak {so}}_{n}}
\newcommand {\sovm}{{\mathfrak {so}}_{8}}
\newcommand {\sosm}{{\mathfrak {so}}_{7}}

\newcommand {\sfr}{\eus R}
\newcommand {\VV}{\eus V}

\newcommand {\vlb}{\eus V_\lb}
\newcommand {\eus}{\EuScript}

\newcommand {\gS}{{\eus S}}

\newcommand {\ap}{\alpha}

\newcommand {\lb}{\lambda}
\newcommand {\vp}{\varphi}

\newcommand {\blb}{\boldsymbol{\lambda}}
\newcommand {\bmu}{\boldsymbol{\mu}}
\newcommand {\bpi}{\boldsymbol{\pi}}
\newcommand {\bvp}{\boldsymbol{\varphi}}
\newcommand {\bpsi}{\boldsymbol{\psi}}

\newcommand {\tvp}{\tilde{\varphi}}

\newcommand {\ca}{{\mathcal A}}

\newcommand {\N}{{\mathfrak N}}
\newcommand {\co}{{\mathcal O}}
\newcommand {\cP}{{\mathcal P}}

\newcommand {\BN}{{\mathbb N}}

\newcommand {\BZ}{{\mathbb Z}}
\newcommand {\BP}{{\mathbb P}}

\newcommand {\md}{/\!\!/}

\newcommand {\ad}{{\mathrm{ad\,}}}

\newcommand {\codim}{{\mathrm{codim\,}}}

\newcommand {\hot}{{\mathsf{ht}}}

\newcommand {\Lie}{{\mathrm{Lie\,}}}

\newcommand {\Ima}{{\mathsf{Im}}}

\newcommand {\rk}{{\mathsf{rk}}}
\newcommand {\spe}{{\mathsf{Spec\,}}}

\newcommand {\tri}{{\mathfrak{sl}}_2}

\newcommand {\GR}[2]{{\textrm{{\sf\bfseries #1}}}_{#2}}

\newcommand {\ov}{\overline}

\newcommand {\beq}{\begin{equation}}
\newcommand {\eeq}{\end{equation}}

\renewcommand{\le}{\leqslant}
\renewcommand{\ge}{\geqslant}
\renewcommand{\lg}{\langle}
\newcommand{\rg}{\rangle}
\newcommand{\bbk}{\Bbbk}

\newcommand {\omin}{\co_{\sf min}}
\newcommand {\omen}{\Omega_{\sf min}}
\newcommand {\omco}{\Omega_{\co}}
\newcommand {\bomin}{\ov{\co_{\sf min}}}
\newcommand {\bco}{\ov{\co}}

\newcommand{\odin}{\mathrm{{1}\!\! 1}}

\begin{document}
\setlength{\parskip}{2pt plus 4pt minus 0pt}
\hfill {\scriptsize October 13, 2024} 
\vskip1ex

\title[Projections of nilpotent orbits]{Projections of nilpotent orbits in a simple Lie algebra and 
shared orbits}
\author{Dmitri I. Panyushev}
\address{Independent University of Moscow,
Moscow 119002, Russia}
\email{panyush@mccme.ru}
\thanks{}
\begin{abstract}
Let $G$ be a simple algebraic group and $\mathcal O\subset \g=\Lie G$ a nilpotent orbit. If $H$ is a reductive subgroup of $G$ with $\h=\Lie H$, then $\g=\h\oplus\me$, where
$\me=\h^\perp$. We consider the natural projections 
$\bvp: \bco\to \mathfrak h$ and $\bpsi: \bco\to \mathfrak m$ and two
related properties of $(H, \co)$:

\centerline{
$(\eus P_1)$:   $\bco\cap \me=\{0\}$; \qquad
$(\eus P_2)$:   $H$ has a dense orbit in $\co$.}

\noindent
It is shown that either of these properties implies that $H$ is semisimple.
We prove that $(\eus P_1)$ implies $(\eus P_2)$ for all $\co$ and the converse holds for 
$\mathcal O_{\sf min}$, the minimal nilpotent orbit. If $(\eus P_1)$ holds, then $\bvp$ is finite and 
$[\bvp(e),\bpsi(e)]=0$ for all $e\in\co$. Then $\overline{\bvp(\co)}$ is the closure of a 
nilpotent $H$-orbit $\mathcal O'$. The orbit $\mathcal O'$ is ``shared" in the sense 
of Brylinski--Kostant (1994). We obtain a classification of all pairs $(H,\mathcal O)$ with property 
$(\eus P_1)$ and discuss various relations between $\mathcal O$ and $\mathcal O'$. In particular, we 
detect an omission in the list of pairs of simple groups $(H,G)$ having a shared orbit that was given
by Brylinski and Kostant. It is also proved that $(\eus P_1)$ for $(H,\omin)$ implies that 
$\ov{G{\cdot}\bvp(\omin)}=\ov{G{\cdot}\bpsi(\omin)}$.

\end{abstract}
\maketitle


\section{Introduction}

\noindent
\subsection{} 
Homogeneous symplectic varieties $X=H/Q$ lie at the heart of of representation theory and symplectic
geometry. Marvellously, all coadjoint $H$-orbits $\co'$ are symplectic varieties. If $H$ is a 
simply-connected semisimple Lie group, then each $X$ is a covering of some $\co'$. From many points 
of view, the most interesting $X$ are those whose base $\co'$ is a nilpotent orbit. In a sense, the 
geometry of these ``nilpotent covers" $X$ is the subject of this paper.
One of our motivations came from a desire to better understand work of R.\,Brylinski and B.\,Kostant on
pairs of simple Lie groups having a ``shared" nilpotent orbit. The main theme of~\cite{bk94} is that a 
covering $X$ of a nilpotent $H$-orbit $\co'$ may ``see'' a larger group $G$ as symmetry group of (the 
closure of) of $X$. If this is the case, then $\co'$ is said to be a {\it shared orbit\/} with respect to the pair 
$(H,G)$. This provides actually two nilpotent orbits:  the $H$-orbit $\co'$ and the $G$-orbit $\co$ that is
the base for the dense $G$-orbit in the closure of $X$. The orbits $\co'$ and $\co$ are also said to be {\it shared}. 

The key result of~\cite{bk94} is the determination of all pairs of simple groups $H\subset G$ having a 
shared nilpotent $H$-orbit. However, it is incomplete. In this paper, we point out an omission 
in~\cite{bk94} and provide a detailed classification of the pairs $(H,G)$ having a shared orbit. Moreover, 
for all such pairs, we determine {\bf all} pairs of shared orbits $(\co',\co)$. Our approach is opposite to 
that of~Brylinski--Kostant. We begin with a nilpotent $G$-orbit $\co$ and look for possible subgroups $H$
having a dense orbit in $\co$. Our methods are rather different and the main tool is the study of projections of $\co$ related to $H\subset G$. This article is also a sequel to~\cite{omin}.

\subsection{} Let $G$ be a simple algebraic group with $\Lie G=\g$ and $\N=\N(\g)$ the nilpotent cone in $\g$.
If $H$ is a reductive subgroup of $G$, then the Killing form of $\g$ is non-degenerate on $\h=\Lie H$
and $\g=\h\oplus\me$, where $\me=\h^\perp$. Here $\N\cap\h=\N(\h)$ is the nilpotent cone in $\h$.
The linear action $(H:\me)$ is the {\it isotropy representation\/} of $H$. If $\h=\g^\sigma$ for an involution $\sigma\in\mathsf{Aut}(\g)$, then we say that $(\g,\h)$ is a {\it symmetric pair\/} and $\h$ is a {\it symmetric subalgebra\/} of $\g$. In this case, $\me$ is the $(-1)$-eigenspace of $\sigma$, and we also write $H=G_0$, $\h=\g_0$, and $\me=\g_1$.

For any $G$-orbit $\co\subset\g$, there are natural projections associated with $\h$ and $\me$:
\[
      \bvp: \bco\to \h  \quad \& \quad \bpsi:\bco\to \me .
\]
Here $\bco$ is the closure of $\co$. If $\omin\subset\N$ is the minimal non-trivial $G$-orbit and 
$\h=\g_0$ is symmetric, then these projections have been studied in~\cite{omin}. It was shown that the 
structure of the $G_0$-varieties $\bvp(\omin)$ and $\bpsi(\omin)$ depends essentially on whether
$\omin\cap\g_1$ is empty or not. The presence of $\sigma$ was also used for describing a connection 
between $\bvp(\omin)$, $\bpsi(\omin)$, and the secant variety of $\BP(\omin)$. There are only six
involutions $\sigma$ such that $\g_1\cap\omin=\varnothing$ (i.e., $\g_1\cap\ov{\omin}=\{0\}$), and for 
those $\sigma$, $G_0$ has a dense orbit in $\omin$. 

In this article, the projections $\bvp$ and $\bpsi$ are considered for arbitrary pairs $(H,\co)$, where $H$ 
is reductive and $\co\subset\N$. We are mainly interested in two related properties: 
\begin{itemize}
\item[$(\eus P_1)$] \ \quad $\bco\cap\me=\{0\}$;
\item[$(\eus P_2)$] \quad $H$ has a dense orbit in $\co$.
\end{itemize}

\noindent
By~\cite[Theorem\,4.6]{omin}, if $(\g,\h)$ is a symmetric pair, then 
$(\eus P_1)\Leftrightarrow (\eus P_2)$ for $\omin$. In Section~\ref{sect:P1-P2}, we prove that if 
$(\eus P_1)$ or $(\eus P_2)$ holds for $(H,\co)$, then $H$ must be semisimple. It is proved that
$(\eus P_1)$ implies $(\eus P_2)$ for any $\co\subset\N$ and semisimple $H$. Moreover we establish 
the equivalence of $(\eus P_1)$ and $(\eus P_2)$  for $\omin$.

If $(\eus P_2)$ holds, then the dense $H$-orbit in $\co$ is denoted by $\omco$ (for $\co=\omin$, we 
write $\omen$ in place of $\Omega_{\omin}$). Here $\bvp(\omco)$ is a nilpotent $H$-orbit, 
$\dim\co=\dim\bvp(\co)=\dim\bvp(\omco)$, and 
$[\bvp(e), \bpsi(e)]=0$ for all $e\in \bco$, see Prop.~\ref{thm-6.1}. Furthermore, if $(\eus P_1)$ holds, 
then the morphism $\bvp: \bco\to \bvp(\bco)$ is finite and $\h^{\bvp(e)}\subset \h^{\bpsi(e)}$ for all 
$e\in \bco$, see Prop.~\ref{prop-ast}. In this case, $\bpsi(e)$ is contained in the closure of 
$G{\cdot}\bvp(e)$. We also explain how to compute the degree of the finite map 
$\bvp$ (Section~\ref{subs:deg}).

The complete classification of pairs $(H,\co)$ with property $(\eus P_1)$ is obtained in 
Sections~\ref{sect:table} and~\ref{sect:classif}. If $\bco\cap\me=\{0\}$, then the same holds for any orbit 
$\co'\subset\bco$; in particular, for $\omin$. Therefore, we first provide the list of all pairs $(G,H)$ such 
that $(\eus P_1)$ holds for $(H,\omin)$, see Table~\ref{table:odin}. (Such subgroups $H$ are said to be 
{\it good}.) In Section~\ref{subs:check}, we verify that $(\eus P_1)$ holds for all items in the list.
Afterwards, the completeness of this list is established in Section~\ref{sect:classif}. To obtain the 
classification of all pairs $(H,\co)$ with property $(\eus P_1)$, one has to pick suitable non-minimal 
nilpotent orbits for each good subgroup $H$ from Table~\ref{table:odin}. In most cases, there are no such 
orbits. A few possibilities for such non-minimal orbits are described in Theorem~\ref{thm:non-minim-good} .
Our classification is summarised in Tables~\ref{table:odin} and~\ref{table:dva}.

Section~\ref{sect:complement} provides some complementary results on $\bvp$ and $\bpsi$, if 
$(\eus P_1)$ holds. Using our classification, we prove that the $H$-orbits $\bvp(\omen)\subset\h$ and 
$\bpsi(\omen)\subset \me$ generate the same $G$-orbit in $\g$. Hence 
$\ov{G{\cdot}\bvp(\omin)}=\ov{G{\cdot}\bpsi(\omin)}$. However, this relation fails for the non-minimal 
orbits $\co$ with $(\eus P_1)$. Another observation is that the height of the nilpotent $H$-orbit 
$\bvp(\omco)$ equals the height of the $G$-orbit $G{\cdot}\bvp(\omco)$. (See Section~\ref{subs:nilp} about the {\it height}.)

\subsection{}
Pairs of orbits $\co\subset \N$ and $\co'=\bvp(\omco)\subset \N(\h)$ have been studied in \cite{bk94},
especially, for $\co=\omin$. We say that  
the orbits $(\co, \bvp(\omco))$ are 
{\it shared}. The classification of pairs of simple groups $(H, G)$ such that the orbits
$\omin$ and $\bvp(\omen)$ are shared is presented in~\cite{bk94}, see Corollary~5.4 and Theorem~5.9
therein. Unfortunately, it is not complete. 
The pair $(\GR{B}{3},\GR{B}{4})$, which is item~7 in our Table~\ref{table:odin}, is missing in that classification. 
Probably, the reason for this omission is that the embedding 
$\GR{B}{3}\subset\GR{B}{4}$ is not the obvious one. On the level of Lie algebras, it is given by the chain 
of embeddings $\mathfrak{spin}_7\subset\sovm\subset \mathfrak{so}_9$. For this embedding, the 
homogeneous space $\GR{B}{4}/\GR{B}{3}$ is spherical~\cite[Tabelle\,1]{kr79}. 
But, for the standard embedding $\sosm\subset\sovm\subset\mathfrak{so}_9$, neither $(\eus P_1)$ nor
$(\eus P_2)$ is satisfied; and in this case the homogeneous space $\GR{B}{4}/\GR{B}{3}$ is of 
complexity~1~\cite[Chap.\,3]{diss}. 

It is worth noting that property $(\eus P_1)$ does not occur in~\cite{bk94} and  
results therein are obtained via the machinery of symplectic geometry.
Our approach to the classification and properties of the projection $\bvp$ is mostly invariant-theoretic. To establish 
either presence or absence of $(\eus P_1)$, various ad hoc methods are employed. In particular, we 
make use of Kimelfeld's classification of locally transitive actions on flag varieties~\cite{kim79,kim},
Onishchik's theory of decompositions of Lie groups~\cite{on62,on69}, some properties of sheets in 
$\g$~\cite{schichten,bokr}, and Dynkin's classification of $S$-subalgebras of the exceptional Lie 
algebras~\cite{Dy52}.

\section{Preliminaries}
\label{sect:prelim}
\noindent
The ground field $\bbk$ is algebraically closed and $\mathrm{char\,}\bbk=0$. All groups and varieties are 
assumed to be algebraic, and all algebraic groups are affine. 
Let $\bbk[X]$ denote the algebra of regular functions on a variety $X$. If $X$ is irreducible, then 
$\bbk(X)$ is the field of rational functions on $X$. If $X$ is acted upon by a group $Q$, then $\bbk[X]^Q$ 
and $\bbk(X)^Q$ are the subalgebra and subfield of invariant functions, respectively. For $x\in X$, let 
$Q^x$ denote the {\it stabiliser} of $x$ in $Q$. Then $\q^x=\Lie(Q^x)$.
If $X$ is affine and $Q$ is reductive, then $X\md Q:=\spe(\bbk[X]^Q)$ and $\bpi:X\to X\md Q$ is
the quotient morphism. 

Throughout, $G$ is a simple algebraic group, $B$ is a Borel subgroup of $G$, $T\subset B$ is a maximal 
torus, and $U=(B,B)$. The respective subgroups in a reductive group $H\subset G$ are denoted by 
$B_H, U_H$, etc. Write $\z_\g(\h)$ or $\g^\h$ for the centraliser of $\h$ in $\g$. It is a reductive 
subalgebra of $\g$. For the exceptional simple algebraic groups and their Lie algebras, we adopt the 
convention that $\GR{G}{2}$, $\GR{F}{4}$, and $\GR{E}{n}$ are groups, while $\eus G_2$, $\eus F_4$, 
and  $\eus E_n$ are Lie algebras.

\subsection{Nilpotent orbits}
\label{subs:nilp}
Let $\N/G$ denote the finite set of $G$-orbits in $\N\subset\g$. The partial order in $\N/G$ is defined by 
the closure relation on the orbits. The Hasse diagrams of $\N/G$ for the small rank classical cases can be found in~\cite[Section\,19]{KP82}.
The dense $G$-orbit in $\N$, $\co_{\sf reg}$, is said to be {\it regular}. 
It is the unique maximal element of $\N/G$. The boundary $\N\setminus\co_{\sf reg}$ is irreducible, and 
the dense orbit in it is said to be {\it subregular}, denoted $\co_{\sf sub}$. The unique minimal nonzero $G$-orbit in $\N$ is denoted by $\omin$. Whenever the dependance on $\g$ is essential (helpful),
we write $\N(\g)$, $\omin(\g)$, etc. It is classically known that $\N=\ov{\co_{\sf reg}}$ is normal (Kostant, 1963) and $\ov{\co_{\sf sub}}$ is normal 
(Broer, 1993).

Any nonzero $e\in\N$ is included into an $\tri$-triple $\{e,h,f\}$, see~\cite[Ch.\,3.3]{CM} or 
\cite[Ch.\,6, \S\,2]{t41}, and the semisimple element $h$ is called a {\it characteristic\/} of $e$. The 
{\it height\/} of $e$ is 
\[
    \hot(e)=\max\{k\mid (\ad e)^k\ne 0\} .
\] 
Set $\hot(G{\cdot}e):=\hot(e)$. 
Then $\hot(\omin)=2$ and $\hot(\co_{\sf reg})=2\boldsymbol{c}(\g)-2$, where 
$\boldsymbol{c}(\g)$ is the Coxeter number of $\g$. Hence $2\le \hot(z)\le 2\boldsymbol{c}(\g)-2$
for any $z\in\N\setminus\{0\}$.
The integer $\hot(G{\cdot}e)$ is readily computed 
via the {\it weighted Dynkin diagram\/} (=\,{\it\bfseries wDd}) of $G{\cdot}e$, see \cite[Eq.\,(2.1)]{p99}. 
For classical $G$, there are formulae for the height via partitions, see~\cite[Theorem\,2.3]{p99}.
If $e\in\N(\h)\subset \N(\g)$, then we write
$\hot_H(e)=\hot_H(H{\cdot}e)$ and $\hot_G(e)=\hot_G(G{\cdot}e)$ for the heights in $\h$ and $\g$, respectively. Clearly, $\hot_H(e)\le \hot_G(e)$. 

If $G$ is classical, then $\co(\bmu)$ is the nilpotent orbit corresponding to the partition 
$\bmu$. But one should keep in mind that, for $G=SO_{4m}$ and a {\it very even} $\bmu$, there are 
actually two associated orbits, denoted $\co(\bmu)_{\sf I}$ and $\co(\bmu)_{\sf II}$~\cite[Chap..\,5.1]{CM}. These orbits are also 
said to be {\it very even}. Their union forms a sole $O_{4m}$-orbit. In general, we shall often use
\begin{itemize}
\item formulas for the weighted Dynkin diagram of $\co(\bmu)$ via $\bmu$~\cite[5.3]{CM};
\item formulas for $\dim\co(\bmu)$~\cite[2.4]{KP82}, \cite[Cor.\,6.1.4]{CM};
\item the description of the closure of $\co(\bmu)$, denoted $\ov{\co(\bmu)}$, in terms of $\bmu$~\cite[6.2]{CM}.
\end{itemize}


\subsection{The orbit of highest weight vectors}
\label{subs:HV}
Let $\vlb$ be a simple $G$-module with highest weight $\lb$ (with respect to some choice of 
$T\subset B$) and $\omin(\vlb)$ the orbit of the highest weight vectors. Let $\N_G(\vlb)$ denote the 
{\it null-cone\/} in $\vlb$, i.e., $\N_G(\vlb)=\{v\in \vlb\mid  \{0\}\in \ov{G{\cdot}v}\}$.
Then $\omin(\vlb)$ is the unique nonzero orbit of minimal dimension in $\N(\vlb)$. 
If $v_\lb\in\vlb$ is a highest weight vector (i.e., a $B$-eigenvector), then $\omin(\vlb)=G{\cdot}v_\lb$.
The orbit $\omin(\vlb)$ is stable under homotheties, i.e., the dilation action of $\bbk^\ast$ on $\vlb$, 
and $\BP(\omin(\vlb))$ is the unique closed $G$-orbit in the projective space $\BP(\vlb)$. Hence
\beq   \label{eq:closure-omin}
     \ov{\omin(\vlb)}=\omin(\vlb)\cup\{0\} ,
\eeq 
see~\cite[Theorem~1]{vp72}. We say that $\omin(\vlb)$ is the {\it minimal $G$-orbit}\/ in  $\N_G(\vlb)$.
Let us recall some basic properties of the minimal $G$-orbits established in~\cite[\S\,1]{vp72}:
\begin{itemize}
\item for any $\lb$, the variety $\ov{\omin(\vlb)}$ is normal and $\bbk[\ov{\omin(\vlb)}]=\bbk[\omin(\vlb)]$;
\item the algebra of regular functions $\bbk[\ov{\omin(\vlb)}]$ is $\BN$-graded and the component of 
grade $n$ is the simple $G$-module $(\VV_{n\lb})^*$, the dual of $\VV_{n\lb}$;
\item $\ov{\omin(\vlb)}$ is a {\it factorial\/} variety $\bigl($i.e., $\bbk[\ov{\omin(\vlb)}]$ is a UFD$\bigr)$ if 
and only if $\lb$ is a fundamental weight. (There is also a general formula for the {\it divisor class group\/}
of $\bomin$, $\mathsf{Cl}(\bomin)$, in terms of the expression of  $\lb$ via the fundamental weights.)
\end{itemize}

\noindent
For the adjoint representation, $\N_G(\g)=\N$ and we usually write $\omin$ in place of 
$\omin(\g)$.

\section{Properties $(\eus P_1)$ and $(\eus P_2)$ for nilpotent orbits}
\label{sect:P1-P2}
\noindent
Throughout, $H$ is a reductive subgroup of a simple algebraic group $G$ and $\g=\h\oplus\me$.
For a $G$-orbit $\co\subset\g$, we consider the following two properties of the pair $(H,\co)$:
\begin{itemize}
\item[$(\eus P_1)$] \ \quad $\bco\cap\me=\{0\}$;
\item[$(\eus P_2)$] \quad $H$ has a dense orbit in $\co$.
\end{itemize}

\noindent
As both properties depend only on the identity component of $H$, we may always assume that $H$ is 
connected. In this section, we prove that 
\begin{itemize}
\item \ $(\eus P_1)$ implies $(\eus P_2)$ for the nilpotent orbits in $\g$;
\item \ these two properties are equivalent for $\co=\omin$.
\end{itemize}
Since $\bomin=\omin\cup\{0\}$, property ($\eus P_1$) for $\omin$ means that 
$\omin\cap\me=\varnothing$.

For the symmetric pairs $(\g,\h)$ and $\omin$, the equivalence of $(\eus P_1)$ and $(\eus P_2)$ is proved 
in~\cite[Theorem\,4.6]{omin}. The general setting requires more efforts and a different approach. 

For $e\in \co$, we write $e=a+b$, where $a=\bvp(e)\in \h$ and $b=\bpsi(e)\in\me$. If $H$ has a dense 
orbit in $\co$, then $\omco$ denotes this $H$-orbit. It is clear that $\omco$ is a homogeneous symplectic 
$H$-variety. Write $\omen$ in place of $\Omega_{\omin}$.

\begin{prop}[cf.\,{\cite[Theorem\,6.1]{omin}}]   
\label{thm-6.1}
If\/ $(\eus P_2)$ holds for $H$ and a  $G$-orbit $\co\subset\g$, then 
\begin{enumerate}
\item $\dim \bvp(\co)=\dim\co$;
\item for any $e=a+b\in\ov{\co}$, one has $[a,b]=0$;
\item if\/  $a+b\in\omco$, then $\h^a=\h^e\subset \h^b$ and $\me^b\subset\me^a$;
\item the centraliser of\/ $\h$ in $\g$ is trivial; in particular, $\h$ is semisimple.
\item If $\co\subset\N$, then $a,b\in\N$ for any $e\in\co$. 
Furthermore, if $a+b\in\omco$ and $h_a\in\h$ is a characteristic of $a\in\h$, then $[h_a,b]=2b$
and hence $[h_a,e]=2e$. 
\end{enumerate}
\end{prop}
\begin{proof}
Only the last embedding in (3) does not occur in~\cite{omin}. To prove it, we use the relation 
$[\g,e]\cap\me=\{0\}$ proved therein. We have $[\me^b,e]=[\me^b,a]\subset [\g,e]\cap\me$.  Hence 
$\me^b\subset \me^a$.
\end{proof}
\noindent
Note that $\co$ is {\bf not} assumed to be nilpotent in parts (1)-(4). But if $H$ has a dense orbit in 
a semisimple orbit $\co$, then it follows from \cite{L72} that $H$ acts transitively on $\co$. This  fits into 
Onishchik's theory of decompositions of semisimple algebraic groups~\cite{on62,on69}, and we do not pursue it here (cf. the proof of Prop.~\ref{diamond-An}).

\begin{cl}   \label{cor:for-P2}
If\/ $(\eus P_2)$ holds for $H$ and $\co\subset\g$, then $\dim \co\le \dim\h-\rk\,\h$.
\end{cl}
\begin{proof} For the $H$-orbit $\bvp(\omco)$,
one has $\dim\co=\dim\bvp(\omco)\le \dim\h-\rk\,\h$.  
\end{proof}

In what follows, it is assumed that $\co\in\N/G$.
Clearly, the orbits $H{\cdot}a\subset\h$ and $H{\cdot}b\subset\me$ do not depend on the choice of 
$e\in\omco$. Then $G{\cdot}a=G{\cdot}\bvp(\omco)$ and $G{\cdot}b=G{\cdot}\bpsi(\omco)$
are the $G$-orbits generated by the projections $\bvp$ and $\bpsi$, respectively. Since
$\dim\bvp(\omco)=\dim\co$, we have $\dim G{\cdot}\bvp(\omco)>\dim\co$. We shall prove in 
Section~\ref{subs:com-pl} that $\ov{G{\cdot}\bvp(\omco)}\supset \co$.

Similar assertions can easily be proved for nilpotent orbits in $\g$ and property $(\eus P_1)$.

\begin{prop}      \label{prop-ast}
Suppose that $(\eus P_1)$ holds for $H$ and $\co\subset \N$. Then
\begin{enumerate}
\item $\bvp:\ov{\co}\to\h$ is finite, hence $ \ov{\bvp(\co)}=\bvp(\bco)$ and $\dim\bvp(\co)=\dim\co$; 
\item for any $e=a+b\in\ov{\co}$, one has\/ $[a,b]=0$ and\/ $\h^a\subset \h^b$;
\item the centraliser of\/ $\h$ in $\g$ is trivial; in particular, $\h$ is semisimple.
\end{enumerate}
\end{prop}
\begin{proof}
 (1)  Here $\bvp^{-1}(\bvp(0))=\bco\cap\me=\{0\}$. Since both $\bco$ and $\ov{\bvp(\co)}$ are
$\bbk^*$-stable (=\,conical), and $\bvp$ is $\bbk^*$-equivariant, this implies 
that $\bvp$ is finite. 
 
(2) By part (1), the fibres of $\bvp$ are finite. On the other hand, for any $e\in\bco$,
\\     \centerline{
$H^a{\cdot}e=a+H^a{\cdot}b$ and $\bvp^{-1}(\bvp(e))\supset a+H^a{\cdot}b$.}
 
\noindent
Hence the orbit $H^a{\cdot}b$ is finite, i.e., $[\h^a, b]=0$. Since 
$a\in\h^a$, this also means that $[a,b]=0$.

(3) If $\bco\cap\me=\{0\}$, then $\bomin\cap\me=\{0\}$ as well. On the other hand,
if $\g^\h\ne \{0\}$, then it contains a nontrivial toral subalgebra of $\g$. Then $\h$ is contained in a 
proper Levy subalgebra of $\g$, and this implies that $\me$ contains a long root space of $\g$, 
cf{.}~\cite[Lemma\,7.1(ii)]{omin}. Hence $\omin\cap\me\ne \varnothing$.
\end{proof}

By Propositions~\ref{thm-6.1}(4) and \ref{prop-ast}(3), dealing with properties $(\eus P_1)$ and 
$(\eus P_2)$ for $\co\in\N/G$, we may assume that $H$ is semisimple.

\subsection{Property $(\eus P_2)$ implies  $(\eus P_1)$ for $\omin$}   \label{subs:P2-to-P1}
Suppose that $(\eus P_2)$ holds for $\omin$ and $\omen=H{\cdot}e$. 
Then $H{\cdot}a=\bvp(\omen)$ is dense in $\ov{\bvp(\omin)}$ and $\dim \omen=\dim\bvp(\omen)$.

It follows from Proposition~\ref{thm-6.1}(3) that, for any $e\in\omen$, we have 
$\bvp^{-1}(\bvp(e))\simeq H^a/H^e$. More precisely, upon the identification of $\h$ and $\h^*$,  
$\bvp\vert_{\omen}$ is the moment map for the homogeneous symplectic $H$-variety $\omen$. Hence $\vp\vert_{\omen}$ is a finite morphism and 
\beq   \label{eq:finite}
\text{ $\bbk[\omen]$ is a finite $\bbk[\bvp(\omen)]$-module.}
\eeq
\begin{lm}     \label{lm:no-divisors} 
If\/ $\g$ is not of type $\GR{A}{n}$, then $\bomin\setminus \omen$ \ contains no divisors.
\end{lm}
\begin{proof}
If $\g$ is neither $\slno$ nor $\spn$, then the highest root is fundamental and hence the divisor class 
group $\mathsf{Cl}(\bomin)$ is trivial, see Section~\ref{subs:HV}. For $\g=\spn$, we have 
$\mathsf{Cl}(\bomin)=\BZ_2$. Therefore, if $\g\ne\slno$, then any divisor 
$D\subset\bomin\setminus \omen$ is given set-theoretically by an equation $D=\{f=0\}$ with 
$f\in\bbk[\bomin]$. Since $H$ 
is semisimple, this implies that $f\in \bbk[\omin]^H$, which contradicts condition $(\eus P_2)$.
\end{proof}

For $\g=\slno$ with $n\ge 2$, we have $\mathsf{Cl}(\bomin)=\BZ$~\cite[Theorem\,5]{vp72}. Hence the proof above does not apply.
However, it appears that the assertion of Lemma~\ref{lm:no-divisors} is valid for $\GR{A}{n}$ as well, see below.

\begin{prop}   \label{prop:ne-An}
If\/ $\g\ne\slno$ and $H$ has a dense orbit in $\omin$, then $\bvp$ is finite and hence 
$\bomin\cap\me=\{0\}$.
\end{prop}
\begin{proof}
Since $\codim (\bomin\setminus\omen)\ge 2$ and $\bomin$ is normal, we have
$\bbk[\bomin]=\bbk[\omen]$.

By assumption, $\ov{\bvp(\omin)}$ is the closure of the nilpotent $H$-orbit $\bvp(\omen)$ and
therefore $\codim \bigl(\ov{\bvp(\omin)})\setminus \bvp(\omen)\bigr)\ge 2$. In this case,
$\bbk[\bvp(\omen)]$ is the integral closure of $\bbk[\ov{\bvp(\omin)}]$ in the fraction field
$\bbk(\ov{\bvp(\omin)})$~\cite[3.7]{bokr}. 

Using Eq.~\eqref{eq:finite} and previous finite extensions,
we conclude that $\bbk[\bomin]$ is a finite $\bbk[\ov{\bvp(\omin)}]$-module and  hence $\bvp$ is finite.
\end{proof}

For the case of $\g=\slno$, more strong and explicit results are obtained.
We use below some results on sheets in semisimple Lie algebras developed in~\cite{schichten, bokr}. 
Recall that a {\it sheet\/} $\eus S\subset \g$ is a maximal irreducible subset consisting of orbits of the 
same dimension. Each sheet is a locally closed subvariety of $\g$ that contains a unique nilpotent orbit~\cite[5.8]{bokr}. 
But it can happen that $\co\subset\N$ belongs to several sheets. A nilpotent orbit is said to be 
{\it rigid}, if the orbit itself is a sheet.

\begin{prop}   \label{diamond-An}
If $H\subset SL_{n+1}$ has property $(\eus P_2)$ for some $\co\subset\N$, then 
$\co=\omin$, $n+1=2m$,  and $H=Sp_{2m}$ is a symmetric subgroup of $SL_{2m}$. Hence $\bvp$ is
finite and $(\eus P_1)$ is also satisfied.
\end{prop}
\begin{proof}
It is known that each nilpotent orbit $\co\subset \N(\slno)$ belongs to a unique sheet, say $\eus S_\co$, 
and that every sheet contains semisimple elements. Let $x\in \eus S_\co$ be semisimple. 

Consider the closed cone $X=\ov{\bbk^*(SL_{n+1}{\cdot}x)}\subset \slno$. We have $\dim X=\dim\co+1$ and 
$X\cap\N=\bco$~\cite[Satz\,3.4]{bokr}. If $H$ has a dense orbit in $\co$, then 
$\max_{x\in X} \dim H{\cdot}x=\dim \co$ and hence $H$ has a dense orbit in every closed orbit 
$SL_{n+1}{\cdot}(\ap x)$ with $\ap\in \bbk^\ast$. Since both $H$ and $(SL_{n+1})^x$ are reductive, the 
$H$-action on the affine variety $SL_{n+1}{\cdot}(\ap x)\simeq SL_{n+1}/(SL_{n+1})^x$ is  transitive~\cite{L72}. 
Therefore, the pair of {\sl reductive} subgroups $(H,(SL_{n+1})^x)$ provides a {\it decomposition\/} of 
$SL_{n+1}$ in the sense of A.L.\,Onishchik. In our case, $(\slno)^x$ is a Levi subalgebra of $\slno$. By 
Onishchik's classification of decompositions~\cite[\S\,4]{on62}, the only possibility is that $(\slno)^x=\gln$, 
$n+1=2m$, and $\h=\mathfrak{sp}_{2m}$. In particular, $(\slno)^x$ is a maximal Levi and the nilpotent 
orbit in the corresponding sheet is $\omin$. Since $H=Sp_{2m}$ is a symmetric subgroup of $SL_{2m}$, 
both properties $(\eus P_1)$ and $(\eus P_2)$ hold for $(H,\omin)$, see~\cite[Section\,5]{omin}. 
\end{proof}
\begin{rema}
In~\cite{on62}, the classification of decompositions is obtained for the compact Lie groups and algebras.
Afterwards, it was shown in~\cite[\S\,2]{on69} that this classification remains the same for the complex 
reductive algebraic groups.
\end{rema}
\noindent
It is known that $\omin$ is rigid unless $\g=\slno$. Hence the previous proof does not work for $\omin$ in 
the other simple Lie algebras. However, the same argument proves the following.

\begin{lm}    \label{lm:sheet}
Let $\gS_\co$ be a sheet of $\g$ that contains a nilpotent $G$-orbit $\co$. If $H\subset G$ has a dense 
orbit in $\co$, then the same is true for any $G$-orbit $\hat\co$ in $\gS_\co$. In particular, if\/ $\gS_\co$ 
contains semisimple elements and $\hat\co=G/K$ is a semisimple orbit in $\gS_\co$, then $HK=G$, i.e.,
$(H,K)$ is a decomposition of $G$ in the sense of Onishchik.
\end{lm}

\begin{ex}   \label{ex:plast}
By Theorem~\ref{thm:non-minim-good} below, the very even orbits $\co(2^{2k})_{\sf I}, \co(2^{2k})_{\sf II}
\subset \mathfrak{so}_{4k}$ have property 
$(\eus P_2)$ with respect to $(G,H)=(SO_{4k},SO_{4k-1})$. If $\co$ is either of them, then the 
semisimple orbits in $\gS_\co$ have stabilisers isomorphic to $GL_{2k}$. (There are two conjugacy classes of Levi subgroups of type $GL_{2k}$ in $SO_{4k}$.)
This means that $(SO_{4k-1}, GL_{2k})$ 
yields a decomposition of $SO_{4k}$. Of course, this decomposition occurs in~\cite{on62}.
\end{ex}

\noindent
It follows from Proposition~\ref{diamond-An} that $\bomin\setminus\omen$ contains no 
divisors for $SL_n$. Hence, {\sl a posteriori}, Lemma~\ref{lm:no-divisors} and 
Proposition~\ref{prop:ne-An} apply to $SL_n$ as well. 
Propositions~\ref{prop:ne-An} \& \ref{diamond-An} yield implication
$(\eus P_2) \Rightarrow (\eus P_1)$ for $\omin$ and all simple Lie algebras.

\subsection{The general implication  $(\eus P_1) {\Rightarrow} (\eus P_2)$}   
\label{subs:P1-to-P2}

\begin{thm}       \label{thm:ast-to-diamond}
Let $\co$ be a nilpotent $G$-orbit. If\/ $\bco\cap\me=\{0\}$, then $H$ has a dense orbit in $\co$.
\end{thm}
\begin{proof}
Assume that $H$ does not have a dense orbit in $\co$. By Rosenlicht's theorem~\cite[1.6]{brion},
this means that $\bbk(\co)^H\ne \bbk$. Since $\bbk(\co)^H$ is the quotient field of 
$\bbk[\co]^H$~\cite[Prop.~2.9]{AP} and $\bbk[\co]$ is the integral closure of $\bbk[\bco]$~\cite[3.7]{bokr}, 
we see that $\bbk[\bco]^H\ne \bbk$. Since $\bvp$ is finite, this implies that
$\bbk[\bvp(\bco)]^H\ne \bbk$ as well. Hence the affine variety $\bvp(\bco)\subset\h$ contains non-trivial closed $H$-orbits.

Therefore, there is $e=a+b\in\bco$ such that $0\ne \bvp(e)=a\in\h$ is semisimple. 
Then $\el:=\h^a$ is a Levi subalgebra of $\h$. Note that $\z_\h(\el):=\h^\el$ is the centre of $\el$ and 
$a\in \z(\el)$ is a generic element of this centre. Consider also $\g^\el$, the centraliser of $\el$ in $\g$. 
Then $\g^\el=\h^\el\oplus\me^\el$ and $e,a,b\in \g^\el$. Moreover, $e$ is a nilpotent element of $\g^\el$. 

Let $Z_G(\el)$ be the connected subgroup of $G$ with Lie algebra $\z_\g(\el)=\g^\el$. Then 
$Z_G(\el){\cdot}e\subset \g^\el$ and
$(Z_G(\el){\cdot}e)\cap \me^\el\subset G{\cdot}e\cap \me=\varnothing$. On the other hand, $\h^\el$ is 
toral subalgebra of $\g^\el$. Therefore, every nilpotent $Z_G(\el)$-orbit in $\g^\el$ meets the 
complementary $\h^\el$-stable subspace $\me^\el$. This contradiction proves the assertion.
\end{proof}

\begin{cl}
If\/ $\bco\cap\me=\{0\}$, then $H$ has a dense orbit in any $G$-orbit $\co'\subset\bco$.
\end{cl}

We shall show in Section~\ref{subs:non-min} that there are series of non-minimal orbits 
$\co\in\N/G$ such that $(\eus P_1)$ holds for some $H\subset G$.

Combining Propositions~\ref{prop:ne-An}, \ref{diamond-An}, and Theorem~\ref{thm:ast-to-diamond}, we obtain 

\begin{thm}   \label{thm:main2}
Let $H$ be a connected reductive subgroup of $G$ and\/ $\co\in\N/G$. Then
 \\[.6ex] 
\centerline{
$(\eus P_1)$: \ $\bomin\cap\me=\{0\}$  \  $\Longrightarrow$ \
$(\eus P_2)$: \ $H$ has a dense orbit in $\omin$.}
For $\co=\omin$, the implication $(\eus P_2) \Rightarrow (\eus P_1)$ is also true. If either of these properties holds, then $H$ is necessarily semisimple.
\end{thm}

\noindent
{\bf Remark.} It is not yet clear whether $(\eus P_2)$ implies $(\eus P_1)$ in general.

\subsection{Computation of $\deg\bvp$}   
\label{subs:deg}
If $\bco\cap\me=\{0\}$, then $\bvp$ is a finite morphism and the integer $\deg\bvp$ is well-defined. It 
is proved in \cite[Prop.\,5.10]{omin} that $\deg\bvp=2$ for the symmetric pairs and $\co=\omin$.
More precisely, in that case $\bvp(\bomin)$ is always normal and $\bvp$ is the quotient map w.r.t. 
the group $\BZ_2$ generated by the involution $\sigma\in GL(\g)$.

\begin{prop}          \label{prop:degree}
If \ $\bco\cap\me=\{0\}$ and $\omco$ is the dense $H$-orbit in $\co$, then 
\[ 
  \deg\bvp=\#\pi_1(\bvp(\omco))/\#\pi_1(\co) , 
\]
where $\pi_1(\cdot)$ stands for the (algebraic) fundamental group. 
\end{prop}
\begin{proof}
Since $\bvp\vert_{\omco}$ is an unramified covering, $\deg\bvp=\#\pi_1(\bvp(\omco))/\#\pi_1(\omco)$. 
On the other hand, the morphism $\bvp$ is finite (Theorem~\ref{thm:ast-to-diamond}) and 
$\bvp(\bco)\setminus \bvp(\omco)$ does not contain divisors. Hence 
$\codim(\co\setminus \omco)\ge 2$ as well. By the purity of branch locus (Zariski--Nagata), we then 
have $\pi_1(\co)=\pi_1(\omco)$.
\end{proof}

The fundamental groups of the nilpotent orbits are known~\cite[Chap.\,6.1]{CM}, and $\deg\bvp$ can 
easily be computed via Proposition~\ref{prop:degree}, see Table~\ref{table:odin} and 
Example~\ref{ex:deg-non-min} below.

\subsection{Commutative planes and property $(\eus P_2)$}   
\label{subs:com-pl}
Suppose that $\co\in\N/G$ and $H$ has a dense orbit in $\co$. If $e=a+b\in\omco$, then $a,b\in\N$ and
$[a,b]=0$, see Prop.~\ref{thm-6.1}. 
Hence $\cP:=\langle a,b\rangle$ is an abelian subalgebra of $\g$ and $\cP\subset\N$. We 
say that $\cP$ is a {\it commutative plane related to\/} $(H,\co)$ with property $(\eus P_2)$.

\begin{thm}    \label{thm:P-dense}
If $e=a+b\in\omco$ and $\cP=\lg a,b\rg$ is the related commutative plane, then $b\in\ov{G{\cdot}a}$
and $G{\cdot}a\cap\cP$ is 
dense in $\cP$.
\end{thm}
\begin{proof}
Let $\{a,h_a,a'\}\subset \h$ be an $\tri$-triple for $a\in\N(\h)\subset\N(\g)$. Consider the $\BZ$-grading
$\g=\bigoplus_{i\in \BZ} \g(i)$ determined by the characteristic $h_a$, i.e.,
$\g(i)=\{x\in \g\mid [h_a,x]=ix\}$. Then $\g(i)=\h(i)\oplus\me(i)$ for all $i\in \BZ$ and $e\in\h(2)$. By
Prop.~\ref{thm-6.1}(5), we have $b\in\me(2)$ and hence $\cP\subset\g(2)$. Let $G(0)$ be the connected (reductive) subgroup of
$G$ with $\Lie G(0)=\g(0)$. By a standard property of $\tri$-triples, $[\g(0),a]=\g(2)$ and
$G(0){\cdot}a$ is dense in $\g(2)$.
This clearly implies that $b\in\ov{G(0){\cdot}a}$ and $G(0){\cdot}a\cap\cP$ is dense in $\cP$.
\end{proof}

\begin{cl}      \label{cl:inclusion-P}
If $(\eus P_2)$ holds for $(H,\co)$, then $\ov{G{\cdot}a}=\ov{G{\cdot}\bvp(\omco)}\supsetneqq
\bco$ and $\hot(G{\cdot}a)\ge \hot (\co)$.
\end{cl}
\begin{proof}
Since $e\in \cP$, we have $\ov{G{\cdot}a}\supset \bco$ and $\dim G{\cdot}a>\dim H{\cdot}a=\dim\co$.
The second inequality stems from the general fact that if $\ov{\co_1}\supset \co_2$, then
$\hot(\co_1)\ge \hot(\co_2)$.
\end{proof}

It follows that $\cP\cap\bco$ is a proper conical subvariety of $\cP$, hence it is a union of lines. We also 
have $\lg e\rg\subset\cP\cap\bco$. For $\co=\omin$, one can obtain more properties of $\cP$. 
Recall that $\hot(\omin)=2$ and $\hot(z)\ge 2$ for any $z\in\N\setminus\{0\}$.

\begin{prop}       \label{prop:cP}
Let $\cP\subset\g$ be an arbitrary plane. 
\begin{itemize}
\item[\sf (i)] If \ $\cP\cap\bomin$ contains three lines, then $\cP\subset\bomin$;
\item[\sf (ii)] If \ $\cP$ is commutative and $\cP\cap\bomin$ contains two lines, then 
$\cP\subset \N$ and $\hot(x)=2$ for any $x\in \cP\setminus\{0\}$.
\end{itemize}
\end{prop}
\begin{proof}
{\sf (i)} This is a special case of the general assertion on the minimal $G$-orbit in an arbitrary simple 
$G$-module, see~\cite[Lemma\,3.1]{omin}. The key fact is that the ideal of $\ov{\omin(\VV_\lb)}$ in 
$\bbk[\VV_\lb]$ is generated by quadrics. For, if a quadric vanishes on three lines in the plane $\cP$, 
then it is identically zero on $\cP$.

{\sf (ii)}  Let $y_1, y_2$ be two linearly independent vectors in $\cP\cap\bomin$. It is a basis for 
$\cP$ and $[y_1,y_2]=0$. Take any $x=c_1y_1+c_2y_2\in \cP$. Then
\[
   (\ad x)^3=c_1^3(\ad y_1)^3+ 3c_1^2 c_2(\ad y_1)^2(\ad y_2)+3c_1c_2^2(\ad y_1)(\ad y_2)^2+
   c_2^3(\ad y_2)^3 .
\]
Since $y_i\in\omin$, we have $(\ad y_i)^3=0$ and $\Ima (\ad y_i)^2\subset \langle y_i\rangle$ for 
$i=1,2$. Because $[y_1,y_2]=0$, this shows that all summands for $(\ad x)^3$ vanish. Hence $x\in\N$ 
and $\hot(x)=2$.
\end{proof}

\begin{rmk}
Suppose that $\h=\g_0$ is a symmetric subalgebra and $H=G_0$ has a dense orbit in $\omin$.
For $e=a+b\in\omen$ and the commutative plane $\cP=\lg a,b\rg$, there are two lines in
$\cP\cap\bomin$, i.e.,  $\lg e\rg$ and $\lg \sigma(e)\rg$. (Here $\sigma(e)=a-b$.)
Therefore, $\hot (x)=2$ for all $x\in\cP$.
This provides a generalisation (and a shorter proof) of Theorem~6.2(i) in \cite{omin}.

If $\h$ is not symmetric, then it can happen that $\hot(x)\ge 3$ for generic $x\in\cP$, see items 6--10 in
Table~\ref{table:odin} below. In these cases,  $\lg e\rg$ is the only line in 
$\cP\cap\bomin$. However, if $\co\ne\omin$, then the ideal of $\bco$ is not generated by quadrics and 
$\dim \Ima (\ad x)^2>1$ for $x\in \co$. 
Therefore, there is no analogue of Proposition~\ref{prop:cP} for arbitrary $\co$.
\end{rmk}

\section{The pairs $(H, \co)$ with property $(\eus P_1)$} 
\label{sect:table}

\noindent
An orbit $\co\in\N/G$ is said to be {\it good}, if there is a semisimple subgroup $H\subset G$ such that 
$\bco\cap\me=\{0\}$, i.e., $(\eus P_1)$ holds. Then the pair $(H,\co)$ is said to be {\it good}, too. If 
$(H,\co)$ is good and $\co'\subset\bco$, then $(H,\co')$ is good as well. Hence the first step towards a 
classification of all good pairs is to determine the subgroups $H\subset G$ such that $(H,\omin)$ 
is good. The corresponding subgroups $H$ are also said to be {\it good}.

In Table~\ref{table:odin}, we gather certain pairs $(G, H)$ such that $G$ is simple and $H\subset G$ is 
good. An accurate verification that $(H,\omin)$ is good for all items is performed in 
Section~\ref{subs:check}.
We also explain how to determine $\bvp(\omen)$ and $G{\cdot}\bvp(\omen)$. For all pairs $(G,H)$ in the 
table, we determine all non-minimal $\co\in\N/G$ such that $(H,\co)$ is still good~(Section~\ref{subs:non-min}). Then it will be shown in Section~\ref{sect:classif} that Table~\ref{table:odin}
gives the {\bf complete} list of good subgroups $H$. Therefore, results of this section provide
the classification of all good pairs $(H,\co)$ for simple groups $G$.

{\it\bfseries Explanations to Table~\ref{table:odin}.}
Let $\omen$ denote the dense $H$-orbit in $\omin$. Then $\bvp(\omen)\in \N(\h)/H$, 
$\dim\bvp(\omen)=\dim\omin$, and $\tilde\co:=G{\cdot}\bvp(\omen)$ is the $G$-orbit generated by 
$\bvp(\omen)$. Nilpotent orbits of the classical groups are represented by partitions~\cite[Chap.\,5]{CM} 
and, for orbits of the exceptional groups, the standard notation is used, see 
tables in~\cite[Chap.\,8]{CM}. The last column provides the degree of the finite projection 
$\bvp:\bomin\to \bvp(\bomin)\subset\h$. The partitions of minimal orbits are
$(2,1^{n-2})$ for $\sln$, $(2^2, 1^{n-4})$ for $\son$, $(2,1^{2n-2})$ for $\spn$.

\begin{table}[ht]
\caption{The pairs $(G,H)$ having properties $(\eus P_1)$ and $(\eus P_2)$ for $\omin$}
\label{table:odin}
\begin{tabular}{c|c| cc| c|c c|c|}
\rus{N0}  & $G,H$ & $\dim\me$ & $\dim\omin$ & $\bvp(\omen)$ & $\tilde\co$ & $\dim\tilde\co$ 
& $\deg\bvp$ \\ 
 \hline\hline
1& $\GR{F}{4},\GR{B}{4}$ & 16 & 16 & $(2^4,1)$ & \rule{0pt}{2.4ex} $\tilde{\mathsf A}_{1}$ & 22 & 2 \\ 
2& $\GR{E}{6},\GR{F}{4}$ & 26 & 22 & \rule{0pt}{2.6ex} $\tilde{\mathsf A}_{1}$ & $2\mathsf{A}_{1}$ & 32 
& 2 \\ 
3& $\GR{B}{n},\GR{D}{n}$ & $2n$ & $4n{-}4$ & $(3,1^{2n-3})$ & $(3,1^{2n-2})$ & $4n{-}2$ & 2 \\ 
4& $\GR{D}{n+1},\GR{B}{n}$ & $2n{+}1$ & $4n{-}2$ & $(3,1^{2n-2})$ & $(3,1^{2n-1})$ & $4n$ & 2 \\ 
5& $\GR{A}{2n{-}1},\GR{C}{n}$ & $\genfrac{(}{)}{0pt}{}{2n}{2}{-}1$ & $4n{-}2$ & $(2^2,1^{2n{-}4})$ & 
$(2^2,1^{2n{-}4})$ & $8n{-}8$ & 2 \\ \hline
6& $\GR{B}{3},\GR{G}{2}$ & 7 & 8 & \rule{0pt}{2.5ex} $\tilde{\mathsf A}_{1}$ & $(3,2,2)$ & 12 & 1 \\
7& $\GR{B}{4},\GR{B}{3}$ & 15 & 12 &  $(3,2,2)$ & $(3,2^2,1^2)$ & 20 & 2 \\ 
8& $\GR{F}{4},\GR{D}{4}$ & 24 & 16 & $(3,2,2,1)$ & \rule{0pt}{2.4ex} $\tilde{\mathsf A}_{1}{+}\mathsf{A}_{1}$ & 28 & 4 \\ 
9& $\GR{G}{2},\GR{A}{2}$ & 6 & 6 & $\co_{reg}{\sim}(3)$ & ${\mathsf{G}}_{2}(a_1)$ & 10 & 3 \\
10& $\GR{D}{4},\GR{G}{2}$ & 14 & 10 & ${\mathsf{G}}_{2}(a_1)$ & $(3,3,1,1)$ & 18 & 6 \\
11& $\GR{C}{n},\displaystyle \times_{i=1}^k\GR{C}{n_i}$ & $d_k$ & $2n$ & $\displaystyle \prod_{i=1}^k \omin(\GR{C}{n_i})$
& $(2^k,1^{2n{-}2k})$ & $k(2n{-}k{+}1)$ & $2^{k-1}$ \\
\hline
\end{tabular}
\end{table}

For items~1--5, $(\g,\h)$ is a symmetric pair and the embeddings $H\subset G$ are well known. Here
$\me$ is a simple $\h$-module. Although the orbits $\bvp(\omen)$ and $\tilde\co$ in item~5 are given by 
the same partitions, they belong to different Lie algebras.

Items~6--10 are non-symmetric cases. We give details on the embeddings $H\subset G$ and the 
structure of the $H$-module $\me$ in Section~\ref{subs:emb}. Here $\me$ is a simple $\h$-module only 
for {\rus{N0}}6.

For item~11, we have $k\ge 2$, $n=\sum_{i=1}^k n_i$, and $d_k=2(n^2-\sum_{i=1}^k n_i^2)$. Here $\h$ 
is symmetric if and only if $k=2$, and then $\me$ is a simple $\h$-module. 

All these items provide pairs of {\it shared orbits\/} $(\omin, \bvp(\omen))$ in the sense R.K.\,Brylinski 
and B.\,Kostant. However, our item~7 is missing in their classification of shared orbits, 
see~\cite[Corollary\,5.4]{bk94} (resp.~\cite[Theorem\,5.9]{bk94}) for the case in which $\rk\,\h<\rk\,\g$ 
(resp. $\rk\,\h=\rk\,\g$). For this reason, we give in Section~\ref{sect:classif} a detailed classification of the 
good pairs $(H,\omin)$ with $G$ simple. It should also be noted that the main emphasis in~\cite{bk94}
is given to the case in which $\h$ is simple. Therefore, our item~11 only appears in Remark\,4.9 therein.

\subsection{Embeddings}    \label{subs:emb}
The  symmetric case has already been studied in~\cite[Sect.\,5]{omin}, and it is 
proved therein that  $(\eus P_1)$ and $(\eus P_2)$ hold for items 1--5 and 11 (with $k=2$).

To handle the remaining cases in Table~\ref{table:odin}, we need explicit embeddings 
$H\subset G$ and thereby some notation on representations. The fundamental weights of $G$ are 
denoted by $\{\tvp_i\}$, and if $H$ is simple, then the fundamental weights of $H$ are 
$\{\vp_i\}$. The numbering of the simple roots and fundamental weights follows~\cite{t41}.
In particular, $\tvp_1$ is always the highest weight of a simplest representation of $G$.
A  simple $H$-module $\sfr_\lb$ is usually identified with its highest weight $\lb$ via the {\sl multiplicative\/} 
notation for $\lb$ in terms 
of the fundamental weights. For instance, $\vp_j\vp_k+3\vp_i^2$ is a shorthand for 
$\sfr_{\vp_j+\vp_k}+3\sfr_{2\vp_i}$ (the sum of four simple $H$-modules). Write $\odin$ for 
the trivial one-dimensional representation.

Below, we point out the restriction of the simplest representation of $G$ to $H$, $\me$ as $H$-module,
 and the highest root of the simple Lie algebra $\h$.

\begin{tabular}{rllll}
6) &  $(\GR{B}{3},\GR{G}{2})$: \ $\dim(\tvp_1)=7$, &  $\tvp_1\vert_H=\vp_1$, &  $\me=\vp_1$, & $\h=\vp_2$ ;\\
7) &  $(\GR{B}{4},\GR{B}{3})$: \ $\dim(\tvp_1)=9$, & $\tvp_1\vert_H=\vp_3{+}\odin$,   &  $\me=\vp_1{+}\vp_3$, & $\h=\vp_2$  ;\\
8) &  $(\GR{F}{4},\GR{D}{4})$: \ $\dim(\tvp_1)=26$, & $\tvp_1\vert_H=\vp_1{+}\vp_3{+}\vp_4{+}2\odin$,  &
$\me=\vp_1{+}\vp_3{+}\vp_4$, & $\h=\vp_2$ ; \\
9) &  $(\GR{G}{2},\GR{A}{2})$: \ $\dim(\tvp_1)=7$, & $\tvp_1\vert_H=\vp_1{+}\vp_2+\odin$, & $\me=\vp_1{+}\vp_2$,
& $\h=\vp_1\vp_2$  ;\\
10) &  $(\GR{D}{4},\GR{G}{2})$: \ $\dim(\tvp_1)=8$, & $\tvp_1\vert_H=\vp_1{+}\odin$, & $\me=2\vp_1$,
& $\h=\vp_2$  .\\
\end{tabular}

\vspace{.6ex}
\noindent For item~11, $H$ is not simple and $\vp^{(i)}_1$ stands for the first fundamental weights of the $i$-th factor of $H$, $\GR{C}{n_i}$. Then 

\quad 11) \quad $(\GR{C}{n},\bigoplus_{i=1}^k \GR{C}{n_i})$: \ $\dim(\tvp_1)=2n$, \ $\tvp_1\vert_H=\bigoplus_{i=1}^k \vp^{(i)}_1$
\ and \ $\me=\bigoplus_{1\le i<j\le k}\vp^{(i)}_1\vp^{(j)}_1$.

\subsection{Verification}    \label{subs:check}
Below, we demonstrate several techniques to check $(\eus P_1)$ or $(\eus P_2)$ for $(H,\omin)$ in the 
non-symmetric cases. The {\it null-cone\/} in the $H$-module $\sfr$ is denoted by $\N_H(\sfr)$ or 
$\N_H(\lb)$, if $\sfr=\sfr_\lb$. Recall that if $\g$ is classical, then $\co(\bmu)$ is the nilpotent orbit in $\g$ 
corresponding to the partition $\bmu$.

\noindent
\textbullet \quad For item~6 in Table~\ref{table:odin}, we actually describe all good nilpotent orbits.
\begin{lm}       \label{lm:B3-G2}
For  $(G,H)=(\GR{B}{3},\GR{G}{2})$, the only good nilpotent $SO_7$-orbits are $\co'=\co(3,1^4)$ and 
$\omin=\co(2^2,1^3)$. 
\end{lm}
\begin{proof}
Here  $\me=\vp_1$ as $\GR{G}{2}$-module and $\N_H(\me)=\N_{\GR{G}{2}}(\vp_1)$ is 
a 6-dimensional quadric in $\vp_1$. Since $SO_7/\GR{G}{2}$ is a spherical homogeneous space, we 
have $\N\cap\me=\N_H(\me)$~\cite[Thm.\,2.4]{p24}. It is known that 
$\N_{\GR{G}{2}}(\vp_1)\setminus\{0\}=:\mathbb O$ is one $\GR{G}{2}$-orbit. Therefore, 
$SO_7{\cdot}\mathbb O\subset \mathfrak{so}_7$ is the only nilpotent orbit  meeting $\me$. By a general 
property of $H$-orbits in isotropy representations, one has 
$\dim SO_7{\cdot}\mathbb O\ge 2\dim\mathbb O=12$~\cite[Lemma\,2.1]{p24}. Since $\dim \co'=10$, 
$\dim\omin=8$, and $\ov{\co'}=\co'\cup\omin\cup\{0\}$, these two $SO_7$-orbits are good.

The variety $G{\cdot}\me\subset \g\simeq\g^*$ is the image of the moment map for the cotangent bundle 
of $G/H$. A general formula for $\dim(\ov{G{\cdot}\me})$ is found in~\cite[Satz\,7.1]{kn90}. If $H$ is reductive, then it can be 
written as $\dim(\ov{G{\cdot}\me})=2\dim\me-\dim\me\md H$.  In our case, we obtain 
$\dim(SO_7{\cdot}\me)=13$. Therefore, $\dim SO_7{\cdot}\mathbb O=12$ and hence 
$SO_7{\cdot}\mathbb O=\co(3,2,2)$, the only nilpotent orbit of dimension~12. Since $\co(3,2,2)$ covers 
$\co'$ in the poset $\N/SO_7$, all other nilpotent orbits in $\sosm$ are also not good.
\end{proof}

\noindent
\textbullet \quad A rational $H$-algebra $\ca$ over $\bbk$ is called a {\it model algebra}, if every 
irreducible representation of $H$ occurs in $\ca$ exactly once. If $\ca$ is finitely generated, then 
$\dim(\spe\ca)=\dim B_H$ and $B_H$ has finitely many orbits in $\spe\ca$. A standard example of model 
algebras is $\ca=\bbk[H/U_H]$.  An interesting approach to model algebras is developed by 
D.\,Luna~\cite{L07}. 

\begin{prop}   \label{prop:model-alg}
For items~6--8, $\bbk[\bomin]$ is a model algebra for $H$. Therefore, $\omin$ is a spherical $H$-variety, $B_H$ has a dense orbit in $\omin$ and thereby $(\eus P_1)$ also holds. 
\end{prop}
\begin{proof}
In these three cases, formulae of Section~\ref{subs:emb} show that $\g=\h\oplus\me$ is the sum of 
{\bf all} fundamental representations of $H$. Since the graded affine algebra $\bbk[\bomin]$ has no zero 
divisors and is generated by the elements of grade $1$ (i.e., by $\g$), {\bf all} finite-dimensional 
irreducible representations of $H$ occur in $\bbk[\bomin]$. By~\cite[Chap.\,1]{diss}, generic 
stabilisers for the $B_H$-action on an affine $H$-variety $X$
are fully determined by the monoid of $H$-highest weights in $\bbk[X]$. In our cases with 
$X=\bomin$, this monoid consists of {\bf all} $H$-dominant weights. Therefore, the 
$B_H$-action on $\bomin$ has trivial generic stabilisers. We also have
$\dim\omin=\dim B_H$. Hence $B_H$ has a dense orbit in $\omin$.
\end{proof}

\noindent
\textbullet \quad For {\rus N0} 8,\ 9,\,11, we have $\rk\,\g=\rk\,\h$ and the root system of $\h$ contains 
all long roots of $\g$ (w.r.t. a common Cartan subalgebra). This implies that $\omin\cap\me=\varnothing$, 
see~\cite[Example\,7.2]{omin}.

\noindent
\textbullet \quad If there is an intermediate semisimple subgroup $\tilde H$, then  the chain
$G\supset\tilde H\supset H$ can be used for our purposes.

\begin{lm}           \label{lm:D4-G2}
For $(\GR{D}{4},\GR{G}{2})$, we have \ $\bomin\cap\me=\{0\}$. 
\end{lm}
\begin{proof}
Consider the chain $\GR{D}{4}\supset \GR{B}{3}\supset  \GR{G}{2}$, where
the first pair is symmetric and the second pair is item~6 in Table~\ref{table:odin}.  Then
\[
   \mathfrak{so}_8=\mathfrak{so}_7\oplus \me_1= ({\eus G}_2\oplus \me_2)\oplus \me_1 
\]
and $\me=\me_2\oplus \me_1$ is the isotropy representation for $(\GR{D}{4},\GR{G}{2})$. Here $\me_1$ and $\me_2$ are isomorphic 7-dimensional $\GR{G}{2}$-modules.
The projection 
$\bvp:\bomin\to {\eus G}_2$ is the composition of two projections:
\[
   \bomin\stackrel{\bvp_1}{\longrightarrow}\mathfrak{so}_7 \stackrel{\bvp_2}{\longrightarrow} 
   {\eus G}_2 .
\]
The symmetric pair $\GR{D}{4}\supset \GR{B}{3}$ occurs as item~4 in Table~\ref{table:odin}. Therefore,
$\bomin\cap\me_1=\{0\}$, $\bvp_1$ is finite, and $\bvp_1(\bomin)$ is the closure of the nilpotent 
$\GR{B}{3}$-orbit $\co'=\co(3,1^4)$. By Lemma~\ref{lm:B3-G2}, we have $\ov{\co'}\cap\me_2=\{0\}$.
Hence $\bvp_2$ is also finite. Thus, $\bvp$ is finite and $\bomin\cap\me=\{0\}$.
\end{proof}

This completes the verification of goodness for items 6--11 of Table~\ref{table:odin}. Let us explain
how to fill in the other columns of the  table. For items 6--10, the
$H$-orbit $\bvp(\omen)$ is uniquely determined by its dimension, which is equal to $\dim\omin$. To compute the $G$-saturation
$\tilde\co=G{\cdot}\bvp(\omen)$, we consider an $\tri$-triple $\{a,h_a,a'\}\subset \h$ and 
$\ah=\lg a,h_a,a'\rg\simeq\tri$. The {\it\bfseries wDd} of the $H$-orbit $\bvp(\omen)$ allows us to obtain the 
decomposition of $\h$ as $\ah$-module. Using partitions for items 7--9,11 or some other {\sl ad hoc} 
methods for $H=\GR{G}{2}$ (items 6,10), we compute the decomposition of $\me$ as $\ah$-module. 
Then the knowledge of $\g$ as $\ah$-module uniquely determines $\tilde\co$.

\subsection{The non-minimal good orbits} 
\label{subs:non-min}
In order to describe the such orbits, it is natural to begin with the orbits covering $\omin$ in the poset
$\N/G$. For the groups $G$ occurring in Table~\ref{table:odin}, the number of  nilpotent orbits covering 
$\omin$ equals 3 for $\GR{D}{4}$, 2 for $\GR{D}{n}$ ($n\ge 5$) or $\GR{B}{n}$ 
($n\ge 3$), and 1 for the remaining cases. The $\GR{D}{4}$-case is treated separately, because it has 
some other peculiarities. Recall that if $G=SO_{4m}$ and $\bmu$ is very even, then
$\co(\bmu)_{\sf I}$ and $\co(\bmu)_{\sf II}$ are two very even $SO_{4m}$-orbits corresponding to $\bmu$.

For $G=SO_{2n}$, we have $\dim(\tvp_1)=2n$ (the simplest representation) and 
$\dim(\tvp_{n-1})=\dim(\tvp_n)=2^{n-1}$ (the half-spinor representations). For $n=4$, all of them are
8-dimensional, and each of them can be chosen as a simplest representation. The nilpotent orbits 
covering $\omin$ in $\N/SO_8$ are $\co(3,1^5)$, $\co(2^4)_{\sf I}$, and $\co(2^4)_{\sf II}$. 

\begin{lm}   \label{lm:trial}
Let us fix a simplest representation $\tvp_1$ for\/ $\g=\sovm$.

{\sf (i)} If the embedding $H=\GR{B}{3}\subset\GR{D}{4}$ is standard, i.e., $\tvp_1\vert_H=\vp_1+\odin$, 
then the only non-minimal good orbits in\/ $\N(\sovm)$ are $\co(2^4)_{I}$ and $\co(2^4)_{II}$.

{\sf (ii)} If the embedding $H=\GR{B}{3}\subset\GR{D}{4}$ is given by $\tvp_1\vert_H=\vp_3$, the spinor
representation, then the only non-minimal good orbits in\/ $\N(\sovm)$ are $\co(3,1^5)$ and one of the orbits $\co(2^4)_{\sf I,II}$.
\end{lm}
\begin{proof}
In both cases $\h$ is a symmetric subalgebra, because these embeddings of $\sosm$ are conjugate 
under the whole non-connected group ${\sf Aut}(\sovm)$. 

{\sf (i)} Let $\sigma$ be the involution such that $(\sovm)^\sigma=\sosm$. Recall that $\me$ is the 
$(-1)$-eigenspace of $\sigma$. There is a criterion in terms of the {\it Satake diagram\/} of $\sigma$ and 
the {\it weighted Dynkin diagram\/} (=\,{\it\bfseries wDd}) of $\co$ that allows us to decide 
whether $\co\cap\g_1\ne \varnothing$, see~\cite[Remark\,2.3]{omin} and examples below. We refer 
to~\cite[Ch.\.4,\,\S\,4.3]{t41} and \cite[Sect.\,2.1]{omin} for generalities on the Satake diagrams. 
The {\it\bfseries wDd} for $\omin=\co(2^2,1^4)$ and the orbits covering $\omin$ are
\\[1ex]
\centerline{
$(2^2,1^4)$: 
\raisebox{-2.5ex}{\begin{tikzpicture}[scale=0.8, transform shape]  
\node (f)  at (5,0) {0};
\node (g) at (6,0) {1};
\node (h) at (7,.5) {0};
\node (i)  at (7,-.5){0};
\foreach \from/\to in {f/g, g/h, g/i}  \draw[-] (\from) -- (\to);        
\end{tikzpicture} };
\quad
$(3,1^5)$: 
\raisebox{-2.5ex}{\begin{tikzpicture}[scale=0.8, transform shape]  
\node (f)  at (5,0) {2};
\node (g) at (6,0) {0};
\node (h) at (7,.5) {0};
\node (i)  at (7,-.5){0};
\foreach \from/\to in {f/g, g/h, g/i}  \draw[-] (\from) -- (\to);        
\end{tikzpicture} };
\quad
$(2^4)_{\sf I}$: 
\raisebox{-2.5ex}{\begin{tikzpicture}[scale=0.8, transform shape]  
\node (f)  at (5,0) {0};
\node (g) at (6,0) {0};
\node (h) at (7,.5) {2};
\node (i)  at (7,-.5){0};
\foreach \from/\to in {f/g, g/h, g/i}  \draw[-] (\from) -- (\to);        
\end{tikzpicture} };
\quad
$(2^4)_{\sf II}$: 
\raisebox{-2.5ex}{\begin{tikzpicture}[scale=0.8, transform shape]  
\node (f)  at (5,0) {0};
\node (g) at (6,0) {0};
\node (h) at (7,.5) {0};
\node (i)  at (7,-.5){2};
\foreach \from/\to in {f/g, g/h, g/i}  \draw[-] (\from) -- (\to);        
\end{tikzpicture} }.
}
\\[2ex]
The Satake diagram for the standard embedding $\sosm\subset\sovm$ is \quad
\raisebox{-1.9ex}{\begin{tikzpicture}[scale=0.65, transform shape]  
\tikzstyle{every node}=[circle, draw] 
\node (f)  at (5,0) {};
\tikzstyle{every node}=[circle, draw, fill=black!65]
\node (g) at (6,0) {};
\node (h) at (7,.5) {};
\node (i)  at (7,-.5){};
\foreach \from/\to in {f/g, g/h, g/i}  \draw[-] (\from) -- (\to);        
\end{tikzpicture} }\ .
The criterion asserts that $\co\cap\g_1\ne\varnothing$ if and only if the nonzero labels of 
{\it\bfseries wDd} correspond to white nodes of the Satake diagram. Hence 
$\co(3,1^5)\cap\g_1\ne\varnothing$, while the  other three orbits are good. Since the closures of all 
larger orbits contain $\co(3,1^5)$, there are no other good orbits.

{\sf (ii)} The Satake diagram for $\mathfrak{spin}_7\subset\sovm$ is either \ \ 
\raisebox{-1.9ex}{\begin{tikzpicture}[scale=0.65, transform shape]  
\tikzstyle{every node}=[circle, draw] 
\node (h) at (7,.5) {};
\tikzstyle{every node}=[circle, draw, fill=black!65]
\node (f)  at (5,0) {};
\node (g) at (6,0) {};
\node (i)  at (7,-.5){};
\foreach \from/\to in {f/g, g/h, g/i}  \draw[-] (\from) -- (\to);        
\end{tikzpicture} }\ \
or \ \
\raisebox{-1.9ex}{\begin{tikzpicture}[scale=0.65, transform shape]  
\tikzstyle{every node}=[circle, draw] 
\node (i)  at (7,-.5){};
\tikzstyle{every node}=[circle, draw, fill=black!65]
\node (f)  at (5,0) {};
\node (g) at (6,0) {};
\node (h) at (7,.5) {};
\foreach \from/\to in {f/g, g/h, g/i}  \draw[-] (\from) -- (\to);        
\end{tikzpicture} }\ , hence the criterion above yields the conclusion.
\end{proof}
\vspace{.6ex} 

\begin{rmk}   \label{rmk:emb-D4-B3}
All embeddings $\GR{B}{3}\subset\GR{D}{4}$ are equivalent under the group 
$\mathsf{Aut}(\GR{D}{4})$. But this is not the case for the embeddings 
$H=\GR{B}{3}\subset\GR{B}{4}=G$. If $\tvp_1\vert_H=\vp_1+2\odin$, then there are no good nilpotent 
orbits for $(G,H)$. But if $\tvp_1\vert_H=\vp_3+\odin$, then $\omin(\mathfrak{so}_9)$ is good (see 
{\rus{N0}}7 in Table~\ref{table:odin}).  In both cases there is a chain 
$\GR{B}{3}\subset\GR{D}{4}\subset\GR{B}{4}$, where the second inclusion is standard, but the first 
inclusions are different, and this is essential! Probably, this is the reason, why our item~7 was 
overlooked in~\cite{bk94}.
\end{rmk}

\begin{thm}   \label{thm:non-minim-good}
Let $(H, \co)$ be a good pair such that $(G,H)$ occurs in Table~\ref{table:odin}  
and $\co$ is non-minimal. Then the only possibilities for $(G,H,\co)$ are:

{\sf (i)} \ $\bigl(SO_N, SO_{N-1}, \co_k:=\co(2^{2k}, 1^{N-4k})\bigr)$ with $2\le k\le [N/4]$ and $N\ge 8$.
For\/ $N=4k$, the partition $(2^{2k})$ is very even and we obtain two good orbits in $SO_{4k}$:
$\co_{k,{\sf I}}$ and $\co_{k,{\sf II}}$.

{\sf (ii)} \ $(\GR{B}{3}, \GR{G}{2}, \co(3, 1^4))$.
\end{thm}
\begin{proof}
{\sf (i)} \ If $(\g,\h)$ is a symmetric pair, then we apply the above-mentioned criterion with Satake 
diagrams and {\it\bfseries wDd}. For items 1--5 and 11 (with $k=2$), this only singles out the 
orbits $\co_k$ in $SO_{N}$ (items 3 and 4). Let us provide some details for the orthogonal case.

The Satake diagram for $(SO_N, SO_{N-1})$ has only one white node (the leftmost one). For instance,
for $N$ even, it is \rule{0pt}{4.3ex} \ \ 
\raisebox{-1.75ex}{\begin{tikzpicture}[scale=0.65, transform shape]  
\tikzstyle{every node}=[circle, draw, fill=black!65]
\node (c) at (2,0) {};
\node (g) at (6,0) {};
\node (h) at (7,.45) {};
\node (i)  at (7,-.45){};
\tikzstyle{every node}=[circle]
\node (d) at (3,0) {};
\node (e) at (4,0) {$\cdots$};
\node (f)  at (5,0) {};
\tikzstyle{every node}=[circle, draw]
\node (b) at (1,0) {};
\foreach \from/\to in {b/c, c/d, f/g, g/h, g/i}  \draw[-] (\from) -- (\to);        
\end{tikzpicture} } \ .
Computing the {\it\bfseries wDd} corresponding to a partition, one readily verifies that 
$\co(3,1^{N-3})\cap\me\ne \varnothing$ \rule{0pt}{4ex} and $\co_k\cap\me=\varnothing$ for all $k$. The 
description of the orbit closures via partitions shows that 
$\ov{\co_k}=\co_k\cup \co_{k-1}\cup\cdots \cup \co_1\cup\{0\}$. Hence the orbits $\co_k$, 
$1\le k\le [N/4]$, are good. (Cf. the proof of Lemma~\ref{lm:trial}.) Here $\co_1=\omin$. All other 
nilpotent orbits contain $\co(3,1^{N-3})$ in their closures. Therefore, they are not good.
 
The Satake diagrams for items 1,\,2,\,5,\,11 in Table~\ref{table:odin} can be found in~\cite[Example\,5.8]{omin}.

{\sf (ii)} \ It remains to examine items 6--10 and 11 (with $k\ge 3$).

\textbullet\quad The pair $(\GR{B}{3},\GR{G}{2})$ is considered in Lemma~\ref{lm:B3-G2}. This yields
the good orbit $\co(3,1^4)$.

\textbullet\quad $(\GR{G}{2},\GR{A}{2})$: \ if property $(\eus P_2)$ holds for $\co$, 
then $\dim\co\le \dim\h-\rk\,\h=6$ (Corollary~\ref{cor:for-P2}). Since $\dim\omin=6$, the only orbit
having $(\eus P_1)$ is $\omin$, cf. Theorem~\ref{thm:ast-to-diamond}.

\textbullet\quad For items~7,8, and~10, there are the chains
$G\supset H_1\supset H$ such that $H_1$ is a symmetric subgroup of $G$. This allows us to prove that
no non-minimal orbits are good.

{\rus N0}7:  \ $\GR{B}{4}\supset \GR{D}{4}\supset  \GR{B}{3}$; \quad
{\rus N0}8:  \ $\GR{F}{4}\supset \GR{B}{4}\supset  \GR{D}{4}$; \quad
{\rus N0}10: \  $\GR{D}{4}\supset \GR{B}{3}\supset  \GR{G}{2}$.

\noindent
For instance, consider the last case. We use the notation of Lemma~\ref{lm:D4-G2}. If 
$\bco\cap\me=\{0\}$, then $\bco\cap\me_1=\{0\}$ as well. By Lemma~\ref{lm:trial}{\sf (i)}, the  
non-minimal good orbits for $(\GR{D}{4},\GR{B}{3})$ are $\co(2^4)_{\sf I}$ and $\co(2^4)_{\sf II}$. 
If $\hat\co$ is either of them, then $\dim\hat\co=12$ and 
$\bvp_1(\Omega_{\hat\co})=\co(3,2,2)\subset \mathfrak{so}_7$ 
(the only nilpotent orbit of dimension~12 in $\mathfrak{so}_7$). By Lemma~~\ref{lm:B3-G2}, $\GR{G}{2}$ 
does not have a dense orbit in $\bvp_1(\hat\co)$. Then $\GR{G}{2}$ does not have a dense orbit in 
$\hat\co$, too. Therefore, $\ov{\hat\co}\cap\me\ne\{0\}$ and there are no non-minimal good orbits.

\textbullet\quad {\rus N0}11:  For $H=\GR{C}{n_1} {\times}\dots {\times} \GR{C}{n_k}$ and $k\ge 3$, there 
is also an intermediate symmetric subgroup 
$H_1=\GR{C}{n_1} \times\GR{C}{n_2+\dots+n_k}$. The only good nilpotent orbit for $H_1$ is
$\omin(\spn)$. Therefore, the same is true for the smaller group $H$.
\end{proof}
The number of good orbits for $(SO_N, SO_{N-1})$ equals $[N/4]$, if $\co_1=\omin$ is also counted.
For all non-minimal good orbits in Theorem~\ref{thm:non-minim-good}, we describe below the 
corresponding orbit in $\h$ and their $G$-saturation. As always, $\Omega_\co$ is the dense $H$-orbit in 
$\co$.

\begin{prop}         
\label{prop:fi-non-min} 
{\sf (i)}  If\/ $\co=\co(2^{2k},1^m)\subset\mathfrak{so}_N$, where $N=4k+m$ and $m>0$, then 
$\bvp(\Omega_\co)=\co(3,2^{2k-2},1^m)\subset\mathfrak{so}_{N-1}$. For $m=0$, one has two good very 
even orbits $\co(2^{2k})_{\sf I}$ and $\co(2^{2k})_{\sf II}$ in $\mathfrak{so}_{4k}$, and the projections of 
both orbits yield the same $SO_{4k-1}$-orbit $\co(3,2^{2k-2})$. In all cases, one has 
$\tilde\co:=SO_N{\cdot}\bvp(\Omega_\co)=\co(3,2^{2k-2},1^{m+1})$, $\hot(\co)=2$, 
and\/ $\hot(\tilde\co)=3$.

{\sf (ii)} \ If\/ $\co=\co(3,1^4)$ in $\mathfrak{so}_7$, then $\bvp(\Omega_\co)=\co_{\sf sub}(\eus {G}_{2})$  
and $G{\cdot}\bvp(\Omega_\co)=\tilde\co=\co(3,3,1)$. Here $\hot(\co)=2$ and $\hot(\tilde\co)=4$.
\end{prop}
\begin{proof}
{\sf (i)}   The formulae for $\dim\co(\blb)$ via $\blb$~\cite[Chap.\,6]{CM} show that 
\[
    \dim\co(2^{2k},1^m)=\dim\co(3,2^{2k-2},1^m)=2k(2k+m-1) . 
\]
Therefore, it suffices to point out $e\in \co\subset \mathfrak{so}_N$ such that
$\bvp(e)\in\mathfrak{so}_{N-1}$ has the required partition. This can be done via the use of the matrix 
model of orthogonal Lie algebras as skew-symmetric matrices with respect to the anti-diagonal.
The assertion about the partition of $\tilde\co$ stems from the general fact that if $\co=\co(\bmu)\subset \mathfrak{so}_{N-1}$ and the embedding $\mathfrak{so}_{N-1}\subset\mathfrak{so}_{N}$ is standard, then the orbit
$SO_N{\cdot}\co$ corresponds to the partition $\tilde{\bmu}=(\bmu,1)$.

{\sf (ii)} Here $\dim\co=10$ and  $\co_{\sf sub}(\eus {G}_{2})$ is the only nilpotent orbit of dimension 10. To compute the $SO_7$-saturation of $\bvp(\omco)=\co_{\sf sub}(\eus {G}_{2})$, we include $e'\in\bvp(\omco)$ into $\tri\subset \eus {G}_{2}$ 
and  consider $\eus {G}_{2}$ and $\me=\vp_1$ as $\tri$-modules. If $\sfr_j$ denotes the simple 
$\tri$-module of dimension $j+1$ (hence $\sfr_0=\odin$), then 
\[
     \eus {G}_{2}\vert_{\tri}=\sfr_4+3\sfr_2 \text{ and } \ \me\vert_{\tri}=2\sfr_2+\sfr_0 .
\]
Then $\sosm\vert_{\tri}=\sfr_4+5\sfr_2+\sfr_0$. Hence $\dim SO_7{\cdot}e'=14$ and this orbit is 
$\co(3,3,1)$.
\end{proof}

\begin{table}[ht]
\caption{The non-minimal good orbits}
\label{table:dva}
\begin{tabular}{r|c| c c c |cc|}    \hline
& $G,H$ & $\co$ & $\bvp(\omco)$ & $\dim\co$ & $\tilde\co$ & $\dim\tilde\co$ \\ \hline
{\sf (i)} & $SO_{N}, SO_{N{-}1}$ & $(2^{2k},\!1^m)$ & $(3,2^{2k-2},\!1^m)$ & $2k(2k{+}m{-}1)$ & 
$(3,2^{2k-2},\!1^{m+1})$   & $2k(2k{+}m)$
\\ 
{\sf (ii)} & $\GR{B}{3},\GR{G}{2}$ & $(3,1^4)$ & $\co_{\sf sub}(\eus G_2)$ & 10 & $(3,3,1)$ & $14$ \\ \hline
\end{tabular}
\end{table}

\begin{rmk}  \label{rem:levass-smith}
The orthogonal series of shared orbits in {\sf (i)} seems to be new. However, for the embedding 
$\GR{G}{2}\subset \GR{B}{3}$, the pairs of orbits $(\omin(\sosm), \tilde{\mathsf A}_1)$ and 
$(\co(3,1^4), \mathsf G_2(a_1))$  have first been discovered by Levasseur and Smith, see~\cite[2.6]{ls88}. 
Since $\tilde{\mathsf A}_1$ is the only orbit of dimension~8 in $\eus G_2$, we may write $\co_8$ for it. It is proved in~\cite{ls88} that $\ov{\co_8}$ is not normal
and $\bvp:  \ov{\omin(\sosm)}\to \ov{\co_8}$ is the normalisation. Historically, it was the first example of
a nilpotent orbit with non-normal closure in an exceptional Lie algebra. Such examples in classical Lie 
algebras has been found earlier in~\cite{KP82}.
\end{rmk}

\begin{ex}   \label{ex:deg-non-min} 
Using Proposition~\ref{prop:degree}, we compute $\deg\bvp$ for the good non-minimal orbits. 
If $G$ is classical, then the formulae for $\pi_1(\co)$ via partitions are given in \cite[Ch.\,6.1]{CM}. For 
the exceptional $G$, consult tables in Chapter\,8 of \cite{CM}.

{\sf (i)} \ For the orbits in Proposition~\ref{prop:fi-non-min}{\sf (i)}, we obtain the following. \\
Let $m$ be odd, hence $G=\GR{B}{n}$ and $H=\GR{D}{n}$, where $n=2k+[m/2]$. 
\begin{itemize}
  \item if $m\ge 3$, then $\pi_1(\co)=\{1\}$, \ $\pi_1(\bvp(\omco))=\BZ_2$, \ and $\deg\bvp=2$.
  \item if $m=1$, then $\pi_1(\co)=\BZ_2$, \ $\pi_1(\bvp(\omco))=\BZ_2^2$, \ and $\deg\bvp=2$.
\end{itemize}
     Let $m$ be even, hence $G=\GR{D}{n}$ and $H=\GR{B}{n-1}$.
\begin{itemize}
  \item if $m\ge 2$, then $\pi_1(\co)=\{1\}$, \ $\pi_1(\bvp(\omco))=\BZ_2$, \ and $\deg\bvp=2$.
  \item if $m=0$, then $\pi_1(\co)=\BZ_2$ for both orbits, \ $\pi_1(\bvp(\omco))=\BZ_2$, \ and 
$\deg\bvp=1$.
\end{itemize}
Thus, we have $\deg\bvp=1$ only if $m=0$; otherwise, $\deg\bvp=2$.

{\sf (ii)} \ For the orbits in Proposition~\ref{prop:fi-non-min}{\sf (ii)}, we have $\pi_1(\co(3,1^4))=\BZ_2$ 
and $\pi_1(\co_{\sf sub}(\eus {G}_{2}))$ is the symmetric group of order 6. Hence $\deg\bvp=3$.
\end{ex}
An interesting phenomenon occurs if $m=0$ in {\sf (i)}. Since $SO_4$ is not simple, we assume that $k\ge 2$. 
The two good very even orbits in $\g=\mathfrak{so}_{4k}$ are $\co_{\sf I}:=\co(2^{2k})_{\sf I}$ and 
$\co_{\sf II}:=\co(2^{2k})_{\sf II}$. Here
$\bvp(\ov{\co_{\sf I}})=\bvp(\ov{\co_{\sf II}})=\ov{\co(3,2^{2k-2})}$ and both projections are birational.
It follows from~\cite{he79} that  $\ov{\co_{\sf I}}$ and $\ov{\co_{\sf II}}$ are normal. But the orbit 
$\co(3,2^{2k-2})$ has a {\it minimal degeneration of type $e$} and therefore $\ov{\co(3,2^{2k-2})}$ is 
not normal~\cite[Theorem\,16.2]{KP82}. Hence either of projections
$\bvp: \ov{\co_{\sf I}} \to \ov{\co(3,2^{2k-2})}$ and 
$\bvp: \ov{\co_{\sf II}} \to \ov{\co(3,2^{2k-2})}$ is the normalisation. This is similar to the phenomenon
discovered in \cite{ls88}, see Remark~\ref{rem:levass-smith}. Here the involution $\sigma$ defining 
$\mathfrak{so}_{4k-1}$ inside $\mathfrak{so}_{4k}$ permutes $\co_{\sf I}$ and $\co_{\sf II}$. Hence the 
$O_{4k}$-orbit $\co_{\sf I\& II}:=\co_{\sf I}\cup\co_{\sf II}$ is $\sigma$-stable and the projection
$\tilde\bvp: \ov{\co_{\sf I\& II}}\to \ov{\co(3,2^{2k-2})}$ has degree 2.
\\ \indent
For $m>0$, the good orbit $\co=\co(2^{2k},1^m)$ is $\sigma$-stable and $\ov{\co(3,2^{2k-2},1^m)}$ 
is normal~\cite{KP82}. This implies that the projection $\bvp: \bco\to \bvp(\bco)$
is the quotient map with respect to the group $\langle\sigma\rangle\simeq \BZ_2$.

\section{Classification}      \label{sect:classif}
\noindent
In this section, we prove that if $(H, \omin)$ is good, then the pair $(G,H)$ is contained in Table~\ref{table:odin}. This also 
means that Table~\ref{table:odin}, Theorem~\ref{thm:non-minim-good}, and 
Proposition~\ref{prop:fi-non-min} provide the complete list of good pairs $(H,\co)$ and shared orbits 
related to the simple Lie algebras $\g$. By Theorem~\ref{thm:main2}, we may consider either of the 
properties $(\eus P_1)$ and $(\eus P_2)$.

A subalgebra $\h$ of $\g$ is called {\it regular}, if it is normalised by a Cartan subalgebra of $\g$. Then 
the root system of $\h$ is a sub-root system of $\g$. A simple observation is that if $(H,\omin)$ is good 
and  $\h$ is regular, then $\h$ must contain all long roots of $\g$ (otherwise, 
$\omin\cap\me\ne\varnothing$). In particular, if $\g$ is simply-laced, then $\h$ cannot be regular. For the
non-simply-laced cases, all possible good pairs with regular $\h$ are given by items 1,\,3,\,8,\,9,\,11 
in Table~\ref{table:odin}.

\subsection{The classical groups $G$} \label{subs:class}
\leavevmode\par
{\bf (I)} \ The series $\GR{A}{n}$ is considered in Proposition~\ref{diamond-An}.

{\bf (II)} For $Sp_{2n}=Sp(\VV)$, $\bomin$ is the categorical quotient of $\VV=\VV_{\tvp_1}$ by $\BZ_2$
(the nonzero element of $\BZ_2$ acts on $\VV$ as $-1$). Therefore, $H$ has a dense orbit in
$\omin$ if and only if $H$ has a dense orbit in the $Sp_{2n}$-module $\VV$.  
By~\cite[Theorem\,1.2]{p03}, we have the following result:

{\it If $H\subset Sp(\VV)$ is reductive and\/ $\bbk[\VV]^H=\bbk$, then 
$H=Sp(\sfr_1)\times\dots\times Sp(\sfr_k)$ for the decomposition $\VV=\bigoplus_{i=1}^k \sfr_i$ into the irreducible $H$-submodules.}

\noindent By Rosenlicht's theorem~\cite[1.6]{brion}, $\bbk[\VV]^H=\bbk$ if and only if $H$ has a dense 
orbit in $\VV$, which is exactly what we need here.

{\bf (III)} The orthogonal series $\GR{B}{n}$ and $\GR{D}{n}$. Now $H\subset G=SO(\VV)$.
Using isomorphisms in small dimensions, we may assume that $\dim\VV\ge 7$. Then 
$\mathfrak{so}(\VV)\simeq \wedge^2\VV=\VV_{\vp_2}$ is a fundamental $SO(\VV)$-module and 
$\BP(\omin)\simeq SO(\VV)/P_2$, where $P_2$ is the maximal parabolic subgroup corresponding to 
the simple root $\ap_2$.

If $H$ has a dense orbit in $\omin$, then $H$ has also a dense orbit in the flag variety $SO(\VV)/P_2$. 
A classification of reductive subgroups of orthogonal groups having a dense orbit in the flag varieties 
$SO(\VV)/P_k$, $1\le k\le [(\dim\VV)/2]$, is obtained by Kimelfeld~\cite{kim79,kim}. However,
it can happen that $H$ has a dense orbit in $\BP(\omin)$, but not in $\omin$. Therefore, our task is to
pick the cases with $k=2$ in Kimelfeld's classification and exclude those of them for which $H$ does not 
have a dense orbit in $\omin$.

$(\boldsymbol{\ap})$: If $\VV$ is a simple $H$-module, then the subgroups $H\subset SO(\VV)$
having a dense orbit in $\BP(\omin)$ are contained 
in~\cite[Table\,1]{kim79}. We meet here our item~6, the embedding $Spin_7\subset SO_8$, which is 
equivalent to our item~4 with $n=3$ (cf. Lemma~\ref{lm:trial}), and three other cases:

\begin{tabular}{lrll}
(1) &  $SL_2\subset SO_{5\ }$, & \ $\tvp_1\vert_H=\vp^4$, & \ $\me=\vp^6$ . \\
(2) &  $Spin_9\subset SO_{16}$,  & \ $\tvp_1\vert_H=\vp_4$, & \ $\me=\vp_3$ .\\
(3) &  $Sp_{2s}\times SL_2\subset SO_{4s}$, & \  $\tvp_1\vert_H=\vp_1\vp'$, 
& \ $\me=\vp_2{\vp'}^2$ ; \\
\end{tabular}
\\[.7ex]
\textbullet \ \ Case (1) is discarded because $4=\dim\omin>\dim SL_2=3$. 

\noindent
\textbullet \ \ For (2):  Now $H=\GR{B}{4}$ and $G=\GR{D}{8}$. 
The highest root of $\GR{D}{8}$ is $\tvp_2$, hence the component of grade $n$ in $\bbk[\bomin]$ is 
$\bbk[\bomin]_n=\VV_{n\tvp_2}$, i.e., $\tvp_2^n$ in our notation. Our goal is show that 
$\bbk[\bomin]^H\ne \bbk$. More precisely, we shall prove that
$\dim (\tvp_2^2)^H=1$. To this end, we need dimension of the subspace of $H$-invariants in some other $G$-modules. 

\noindent 
{\bf --} \ Since $\tvp_1\vert_H=\vp_4$ is a simple orthogonal $H$-module, the symmetric square $S^2\tvp_1=\tvp_1^2\oplus\odin$ contains a unique $H$-invariant. Therefore,  $(\tvp_1^2)^H=\{0\}$.

\noindent 
{\bf --} \ $\tvp_2\vert_H=\wedge^2(\tvp_1\vert_H)=\wedge^2\vp_4=\vp_2+\vp_3$ (the sum of two 
orthogonal $H$-modules). It follows that $S^2\tvp_2=\odin+\tvp_1^2+\tvp_2^2+\tvp_4$ contains only 
two linearly independent $H$-invariants and $\dim(\tvp_2^2+\tvp_4)^H=1$. 
Since $\tvp_4=\wedge^4\tvp_1$, we have $\tvp_4\vert_H=\wedge^4 (\tvp_1\vert_H)=\wedge^4\vp_4$. 
Therefore, it suffices to prove that $(\wedge^4\vp_4)^H=\{0\}$.

\noindent 
{\bf --} \ To any orthogonal $H$-module $\VV$, one may associate the $H$-module
$\mathsf{Spin}(\VV)$ related to the exterior algebra of $\VV$~\cite[Section\,2]{p01}. The orthogonal
$H$-module $\vp_4$ has no zero weight. Hence using~\cite[Prop.\,2.4(ii)]{p01}, we obtain
\[
         \wedge^\bullet\vp_4={\mathsf{Spin}}(\vp_4)\otimes {\mathsf{Spin}}(\vp_4) .
\]
By \cite[Example\,5.7(2)]{p01}, one has ${\mathsf{Spin}}(\vp_4)=\vp_1^2+\vp_3+\vp_1\vp_4$ (the sum 
of three orthogonal $H$-modules). Therefore, $\dim (\wedge^\bullet\vp_4)^H=3$. Since
$\wedge^i\vp_4\simeq \wedge^{16-i}\vp_4$, the $H$-invariants occur only in degrees $i=0,\,8,\,16$.

\noindent 
{\bf --} \ Thus, we obtain that $\dim(\tvp_4)^H=0$ and hence $\dim(\tvp_2^2)^H=1$. Therefore, $H$ does not have a dense orbit in $\omin$.

\noindent
\textbullet \ \ For (3), the similar method works. One also proves here that $\dim \bbk[\omin]^H_2=1$. 

$(\boldsymbol{\beta})$: If the $H$-module $\VV$ is the sum of two totally isotropic $H$-invariant 
subspaces, then for all cases with $k=2$ we have $\g^\h\ne \{0\}$, see~\cite[Table\,4]{kim79}. Hence 
there cannot be a dense $H$-orbit in $\omin$.

$(\boldsymbol{\gamma})$: Suppose that $\VV=\VV_1\oplus\VV_2$ is the sum of orthogonal 
$H$-modules and $d_1=\dim\VV_1\le \dim\VV_2=d_2$.
\\
{\bf --} \ If $d_1\ge 2$, then~\cite[Table\,7]{kim79} shows that either $H=H_1$ is symmetric, or there is an intermediate symmetric subgroup $H_1$, i.e., $H\subset H_1\subset SO(\VV)$. In all these cases,
the pair $(H_1,\omin)$ is already not good.
\\
{\bf --} \ If $d_1=1$, then we meet in~\cite[Table\,8]{kim79} our items 3,\,4,\,7,\,10 and some other cases, which are easily seen to be bad. For instance, the case 
$H=Spin_9\subset SO_{16}\subset SO_{17}=G$ with 
$\tvp_1\vert_H=\vp_4+\odin$ is clearly worse than the bad case (2) in part~$(\boldsymbol{\ap})$. 

This completes our classification for the orthogonal groups.

\subsection{The exceptional groups $G$}  
\label{subs:except} 
As explained above, if $\h$ is regular and $(H,\omin)$ is good, then $\g$ has roots of different length
and $\h$ contains all long roots of $\g$.
This yields items~1,\,8,\,9 in Table~\ref{table:odin}. Following E.B.\,Dynkin, we say that $\h$ is an
$S$-{\it subalgebra\/} of $\g$, if it is not contained in a proper regular subalgebra.
A classification of the $S$-subalgebras of the exceptional Lie algebras and the inclusion relation between 
them is obtained in~\cite[\S\,14]{Dy52}. In particular, the maximal $S$-subalgebras are listed in
Theorem~14.1 in loc.\,sit., see also~\cite[Chap.\,6, \S\,3.3]{t41}. 

\textbullet \ \ Most of the maximal $S$-subalgebras $\h$ are not good, because 
$\dim\h-\rk\,\h <\dim\omin$. This already shows that there are no non-regular good subgroups in 
$\GR{G}{2}$ and $\GR{F}{4}$. 
There remain only {\bf three} cases with $\dim\h-\rk\,\h \ge\dim\omin$ in series $\GR{E}{n}$. 

\textbullet \ \ The symmetric subalgebra $\eus F_{4}$ in $\eus E_6$ is good
(item~2 in Table~\ref{table:odin}), but all smaller $S$-subalgebras of $\eus E_6$ inside $\eus F_{4}$ are already too 
small to be good. This settles the problem for $G=\GR{E}{6}$.

\textbullet \ \ The other two cases are $\GR{A}{1}{\times}\GR{F}{4}\subset \GR{E}{7}$ and 
$\GR{G}{2}{\times}\GR{F}{4}\subset \GR{E}{8}$. For the first of them, 
one can prove that $\dim\bbk[\bomin]^H_2=1$ using the same method as in Section~\ref{subs:class} for
$Spin_9\subset SO_{16}$. The last case is most difficult, and we provide a thorough treatment. Here 
$\dim\h-\rk\,\h=60$ and $\dim\omin(\eus E_8)=58$.

Suppose that $H=\GR{G}{2}{\times}\GR{F}{4}$ has a dense orbit $\omen$ in 
$\omin=\omin(\eus E_8)$. Then $\bvp(\omen)$ is a nilpotent $H$-orbit of dimension 58. Therefore, 
for dimension reason, $\bvp(\omen)$ equals either
$\co_{\sf sub}(\eus G_{2})\times \co_{\sf reg}(\eus F_{4})$ or
$\co_{\sf reg}(\eus G_{2})\times \co_{\sf sub}(\eus F_{4})$. Recall that $\co_{\sf reg}(\cdot)$
(resp. $\co_{\sf sub}(\cdot)$) is the unique {\it regular\/} (resp. {\it subregular}) nilpotent orbit and then
$\ov{\co_{\sf reg}(\cdot)}=\N(\cdot)$.

{\bf -- } Assume that the {\sl first\/} possibility holds, i.e.,  
$\bvp(\bomin)= \ov{\co_{\sf sub}(\eus G_{2})}\times \N(\eus F_{4})$. Here 
$\bvp(\bomin)$ is normal and the morphism $\bvp$ is finite,
see Section~\ref{subs:P2-to-P1}.
Consider the categorical quotients of $\bomin$ and $\bvp(\bomin)$ w.r.t. $\GR{G}{2}\subset H$. Clearly, 
the closed $\GR{G}{2}$-orbits in $\bvp(\bomin)$ are fixed points, hence the same is true for $\bomin$.
Therefore, $\bomin\md \GR{G}{2}\simeq (\bomin)^{\GR{G}{2}}$ and likewise for $\bvp(\bomin)$. 
We obtain the commutative diagram
\[
  \xymatrix{  & \bomin \ar[r]^(.31){\bvp} \ar[d]^{\bpi_1}   & 
  \ov{\co_{\sf sub}(\eus G_{2})}{\times} \N(\eus F_{4}) \ar[d]^{\bpi_2} \ar@{=}[r] & 
  \bvp(\bomin) \ar[d]^{\bpi_2}  \\
  (\bomin)^{\GR{G}{2}} \ar@{=}[r]^-{\sim} &  \bomin\md \GR{G}{2} \ar[r]^{{\bvp\md {\bf G}_2}}  
  &  \N(\eus F_{4})   \ar@{=}[r]^-{\sim} &   \bvp(\bomin)^{\GR{G}{2}}
   }
\]
Here $\bpi_1$ and $\bpi_2$ are quotient maps with respect to $\GR{G}{2}$, and the induced map 
$\bvp\md\GR{G}{2}$ is again finite.

For $\g=\eus E_{8}$, one has $\tvp_1=\g$ and $\tvp_1\vert_H=(\vp_2+\vp'_4)+\vp_1\vp'_1=
(\eus G_2+\eus F_4)+\me$. Here $\vp_1,\vp_2$ (resp. $\vp'_1, \dots,\vp'_4$) are the fundamental weights of $\GR{G}{2}$ (resp. $\GR{F}{4}$).
By definition, $\bvp$ is the restriction to $\bomin$ of the 
projection of $H$-modules 
\[
  \g= (\vp_2+\vp'_4)+\vp_1\vp'_1\to \vp_2+\vp'_4=\h .
\]
The $\GR{G}{2}$-fixed points in $\g$ form the $H$-submodule $\vp'_4$, hence 
$(\bomin)^{\GR{G}{2}}=\bomin\cap \vp'_4$. Since $ (\bomin)^{\GR{G}{2}}\simeq
\bomin\md \GR{G}{2}$, the closure of any $\GR{G}{2}$-orbit contains a unique fixed point. Set
$x=x_1+x_2+x_3\in \bomin$, where $x_1\in \vp_2, \ x_2\in \vp'_4, x_3\in \vp_1\vp'_1$. As the 
$\GR{G}{2}$-action does not affect the component $x_2$, the only possible fixed point in
$\ov{\GR{G}{2}{\cdot}x}$ is $x_2$. Hence $\bpi_1$ can be identified with the projection
from $\bomin$ to $\vp'_4$. Likewise, $\bpi_2$ is identified with the projection from $\bvp(\bomin)$ to 
$\vp'_4$. Thus, both $\bpi_1$ and $\bpi_2\circ\bvp$ are projections to $\vp'_4$, and hence
$\bvp\md \GR{G}{2}$ is an isomorphism.

Since $\omin$ is smooth, the same is true for $(\omin)^{\GR{G}{2}}$~\cite{fogarty}. Hence
the origin is the only singular point in $(\bomin)^{\GR{G}{2}}$. On the other hand, the set of singular points in $\N(\eus F_{4})$ equals $\ov{\co_{\sf sub}(\eus F_{4})}$, which is
of codimension~2. A contradiction! 

{\bf --} A similar argument shows that the case $\N(\eus G_{2})\times \ov{\co_{\sf sub}(\eus F_{4})}$ is 
also impossible.

This proves that there are no good subgroups $H$ in $\GR{E}{7}$ or $\GR{E}{8}$ and completes our classification of good pairs $(H,\omin)$ for the exceptional groups.

\section{Some complements and open problems}
\label{sect:complement}

\noindent
For the projections of nilpotent $G$-orbits associated with the sum $\g=\h\oplus\me$, there are several 
interesting properties and observations that are not fully understood yet.

\subsection{} \label{subs:saturation}
By~\cite[Theorem\,6.7]{omin}, if  $(\g,\h)=(\g,\g_0)$ is a symmetric pair, then
\beq   \label{eq:ravno}
    \ov{G{\cdot}\bvp(\omin)}=\ov{G{\cdot}\bpsi(\omin)} .
\eeq
Let $\mathcal Y$ denote this common subvariety of $\g$. The structure of $\mathcal Y$ depends on whether the 
intersection $\omin\cap\g_1$ is empty or not. More precisely,

{\sf (1)} \ \ if $\omin\cap\g_1\ne\varnothing$, then $\mathcal Y=\ov{\bbk^\ast(G{\cdot}h)}$, where $h$ is 
a characteristic of $e\in \omin$; 

{\sf (2)} \ \ if $\omin\cap\g_1=\varnothing$, then $\mathcal Y=\ov{\tilde\co}$, where the orbit
$\tilde\co\in \N/G$ occurs in Table~\ref{table:odin}.

\noindent
However, whereas the proof of {\sf (1)} is conceptual, the proof of {\sf (2)} relies on the classification 
of good symmetric subgroups $G_0$. In case {\sf (2)}, both $G{\cdot}\bvp(\omin)$ and 
$G{\cdot}\bpsi(\omin)$ contain dense $G$-orbits. Hence Equality~\eqref{eq:ravno} means that these 
orbits coincide. Equivalently, if $e=a+b\in\omen\subset\omin$, then $G{\cdot}a=G{\cdot}b$. Surprisingly, 
this remains true for all good pairs of the form $(H,\omin)$. That is, we have

\begin{thm}     \label{thm:6.1}
If \ $\omin\cap\me =\varnothing$, then Equality~\eqref{eq:ravno} holds.
\end{thm}
\begin{proof}
Our proof for the non-symmetric pairs in Table~~\ref{table:odin} is also case-by-case. 
\\ \indent
\textbullet \ For $(G,H)=(\GR{B}{3},\GR{G}{2})$ ({\rus N0}6),  this is basically proved in 
Lemma~\ref{lm:B3-G2}.

\textbullet \ {\rus N0}7, $(\GR{B}{4},\GR{B}{3})$. Here $\dim\tilde\co=20$. Since $b\in\ov{G{\cdot}a}$
(Theorem~\ref{thm:P-dense}), we always have 
\[
      G{\cdot}\bpsi(\omen)\subset \ov{G{\cdot}\bvp(\omen)}=\ov{\tilde\co} .
\] 
Furthermore, $\dim G{\cdot}\bpsi(\omen)\ge 2\dim \bpsi(\omen)$~\cite[Lemma\,2.1]{p24}. Hence 
$\dim \bpsi(\omen)\le 10$. On the other hand, the linear span of $\omin$ equals $\g$ and $\omen$ is 
dense in $\omin$. Hence the linear span of the $H$-orbit $\bpsi(\omen)$ equals $\me=\vp_1+\vp_3$. It 
is easily seen that the minimal $H$-orbit in $\me$ with this property is $H{\cdot}v$, where
$v=v_{\vp_1}+v_{\vp_3}\in \me$ is a sum of highest weight vectors in ${\vp_1}+{\vp_3}$. Since $\dim H{\cdot}v=10$,
we obtain $H{\cdot}v=\bpsi(\omen)$, $\dim G{\cdot}v=20$, and hence $G{\cdot}v=\tilde\co$. 

\textbullet \ The same approach works for items~8 and 9. 
\\ \indent
\textbullet \ For {\rus N0}10, the similar idea has to be slightly adjusted. Here $\me=2\vp_1$ is not a 
multiplicity free $H$-module. Hence the sum of two highest weight vectors 
$v=v^{(1)}_{\vp_1}+v^{(2)}_{\vp_1}$ lies in a proper $H$-invariant subspace of $\me$. A right choice is 
$v=v^{(1)}_{\vp_1}+v^{(2)}_{\vp_1-\ap_1}$. Then $H{\cdot}v$ is the minimal $H$-orbit such that $\lg H{\cdot}v\rg=\me$ and $\dim H{\cdot}v=9=\frac{1}{2}\dim\tilde\co$.
\\ \indent
\textbullet \ Computations for {\rus N0}11 with $k\ge 3$, which are omitted here, exploit the interpretation 
of $\omin(\spn)$ as the set of symplectic matrices of rank~1. We show that, for a suitable choice of
$x\in \omin(\spn)$, the matrix $\bpsi(x)\in\me\subset\spn$ has the property that $\rk\,\bpsi(x)=k$ and
$\bpsi(x)^2=0$. Then one uses the fact that the orbit $\co(2^k,1^{2n-2k})\subset\N(\spn)$ consists of the symplectic 
matrices $\mathcal M$ such that $\mathcal M^2=0$ and $\rk\,\mathcal M=k$.
\end{proof}

Anyway, a conceptual proof of Theorem~\ref{thm:6.1} would be most welcome.
On the other hand, it follows from our classification that if $\co$ is a non-minimal good orbit, then
$\ov{G{\cdot}\bvp(\co)}\ne\ov{G{\cdot}\bpsi(\co)}$. More precisely,

\begin{prop}   \label{net-ravno}
If $(H,\co)$ is good and $\co$ is not minimal,
then $\omin\subset\bco$ is also good and
\[
   \ov{G{\cdot}\bvp(\co)}\supsetneqq \ov{G{\cdot}\bvp(\omin)}=\ov{G{\cdot}\bpsi(\omin)}=
   \ov{G{\cdot}\bpsi(\co)} .
\]
\end{prop}
\begin{proof}
By Theorem~\ref{thm:non-minim-good}, there are only a series of good orbits for $(SO_N,SO_{N-1})$
and a good non-minimal orbit for $(\GR{B}{3},\GR{G}{2})$. In both cases, 
$\N_H(\me)\setminus\{0\}$ is a sole $H$-orbit. (Incidentally, $\N_H(\me)$ is a hypersurface in $\me$.)
Therefore, $\ov{\bpsi(\omin)}=\ov{\bpsi(\co)}=\N_H(\me)$. The equality
$\ov{G{\cdot}\bvp(\omin)}=\ov{G{\cdot}\bpsi(\omin)}$ stems from Theorem~\ref{thm:6.1}. Then
using Proposition~\ref{prop:fi-non-min} and Table~\ref{table:odin}, we conclude that the $H$-orbits
$\bvp(\omen)$ and $\bvp(\omco)$ are different, and
$\bvp(\omen)\subset\ov{\bvp(\omco)}$.
\end{proof}

\subsection{} 
\label{subs:ht}
By ~\cite[Theorem\,6.2]{omin}, if $\g_0\subset\g$ is a good symmetric subalgebra and 
$e=a+b\in\omin$, then $(\ad e)^3=(\ad b)^3=0$. (A simpler proof stems from 
Proposition~\ref{prop:cP}(ii).) In other words, $\hot(\tilde\co)=2$. But this is not true for 
non-symmetric cases of Table~\ref{table:odin}. For items 6--10, $\hot(\tilde\co)$
is equal to $3,3,3,4,4,$ respectively (use formulae for the height in~\cite[Section\,2]{p99}).
Nevertheless, one can notice a nice general property.

For all items in Table~\ref{table:odin}, one readily verifies that the height of the $H$-orbit 
$\bvp(\omen)\subset \h$ equals the height of the $G$-orbit $\tilde\co=G{\cdot}\bvp(\omen)$. This   
also holds for the non-minimal good orbits $\co$ of Proposition~\ref{prop:fi-non-min} and Table~\ref{table:dva}. More precisely,
\\ 
\textbullet \ \ For {\sf (i)},  $\bvp(\omco)=\co(3,2^{2k-2},1^m)\subset \mathfrak{so}_{N-1}$ and
$SO_N{\cdot}\co(3,2^{2k-2},1^m)=\co(3,2^{2k-2},1^{m+1})$, where both heights are equal to $3$ for
$k\ge 2$. 
\\ 
\textbullet \ \ For {\sf (ii)}, the height of both orbits in question is equal to 4. 

\noindent
That is, for all good pairs $(H,\co)$, one has 
\beq             \label{eq:height}
            \hot_H(\bvp(\omco))=\hot_G\bigl(G{\cdot}\bvp(\omco)\bigr) .
\eeq

\noindent
It would be interesting to have a conceptual explanation for~\eqref{eq:height}. Another curious corollary 
of our classification is that, for a good orbit $\co\in\N/G$, one always has $\hot(\co)=2$.

\subsection{}   \label{subs:cP}
Let $\cP=\lg a,b\rg$ be a commutative plane related to a good pair $(H,\co)$. Recall that 
this means that $e=a+b\in\omco\subset\co$. An interesting open problem is to determine all $G$-orbits 
meeting $\cP$. So far, the following orbits have been found (detected):
$G{\cdot}a=\tilde\co=G{\cdot}\bvp(\omco)$, $\co=G{\cdot}(a+b)$, and 
$G{\cdot}b=G{\cdot}\bpsi(\omco)$. 

By Section~\ref{subs:saturation}, if $\co=\omin$, then $G{\cdot}a=G{\cdot}b$; but if $\co\ne\omin$, then $b\not\in G{\cdot}a$.

\begin{lm}     \label{lm:2-orb}
If $\co=\omin$ and $\h$ is a symmetric subalgebra of $\g$, then $\tilde\co$ and $\omin$ are the only nonzero $G$-orbits meeting $\cP$. 
\end{lm}
\begin{proof}
In these cases ({\rus N0}1--5 and {\rus N0}11 with $k=2$), $\tilde\co$ covers $\omin$ in the poset $\N/G$. We have also proved that $\tilde\co\cap\cP$ is dense in $\cP$ (Theorem~\ref{thm:P-dense}).
\end{proof}

If $\co=\omin$ and $\h$ is not symmetric, then there are always intermediate orbits between $\tilde\co$
and $\omin$. In some cases, we are able to prove that some of them do meet $\cP$. 

For $(G,H)=(\GR{C}{n},\displaystyle \times_{i=1}^k\GR{C}{n_i})$ with $k\ge 3$, we have
$ \tilde\co=G{\cdot}\bvp(\omen)= \co(2^k,1^{2n-2k})$
and $\omin=(2,1^{2n-2})$. In this case, the intermediate orbits are $\co_l=\co(2^l, 1^{2n-2l})$ with 
$2\le l\le k-1$, and we can prove that $\co_{k-1}\cap\cP\ne \varnothing$.

For $(G,H)=(\GR{G}{2},\GR{A}{2})$, the intermediate orbit is $\tilde{\eus A}_1=\co_8$. Using direct 
calculations with root system $\GR{G}{2}$, we can produce an explicit model of $\cP\subset\eus G_2$ 
and prove that $\co_8\cap\cP\ne\varnothing$. (This is an advantage of a small rank case.)

\subsection{} 
\label{subs:triple1}
Associated with Table~\ref{table:odin}, there are some curious coincidences. 

\textbullet \ Certain pairs 
of groups from that table can be organised into triples $G\supset H_1\supset H_2$ such that both 
$(G,H_1)$ and $(H_1,H_2)$ occur in Table~\ref{table:odin}, and the following property holds for
the related projections $\bvp_1:\g\to \h_1$ and $\bvp_2:\h_1\to \h_2$:
\[ 
  \bvp_1(\omin(\g))=\tilde\co=H_1{\cdot}\bvp_2(\omin(\h_1)) .
\] 
This equality expresses the coincidence of two nilpotent $H_1$-orbits. Below, we point out the  suitable 
triples together with the corresponding $H_1$-orbit $\tilde\co$.
\begin{center}
\begin{tabular}{|ll||ll|}
$\GR{E}{6}\supset \GR{F}{4}\supset \GR{B}{4}$, & $\tilde{\mathsf A}_{1}$ &
$\GR{D}{n+1}\supset \GR{B}{n}\supset \GR{D}{n}$, & $\co(3,1^{2n-2})$ \\
$\GR{B}{4}\supset \GR{B}{3}\supset \GR{G}{2}$, & $\co(3,2,2)$ \rule{0pt}{2.4ex} & 
$\GR{B}{n}\supset \GR{D}{n}\supset \GR{B}{n-1}$, & $\co(3,1^{2n-3})$ \\
$\GR{D}{4}\supset \GR{G}{2}\supset \GR{A}{2}$, & ${\mathsf G}_{2}(a_1)$ \rule{0pt}{2.4ex} & 
$\GR{A}{2n-1}\supset \GR{C}{n}\supset \GR{C}{n_1}\times \GR{C}{n_2}$, & $\co(2^2,1^{2n-4})$ \\
\end{tabular}
\end{center}

\textbullet \ Another kind of triples $G\supset H_1\supset H_2$ combines pairs of shared orbits of 
into the triple of the form
\[
     \co=\omin(\g)\leadsto \co'=\bvp_1(\Omega_\co)\leadsto \co''=\bvp_2(\Omega_{\co'}) 
\]
such that the orbits $(\co,\co'')$ are also shared. These are
\begin{center}
\begin{tabular}{ll}
$\GR{F}{4}\supset \GR{B}{4}\supset \GR{D}{4}$, & 
$\omin(\eus F_4)\leadsto \co'=\co(2^4,1) \leadsto \co''=\co(3,2,2,1)$; \\
$\GR{D}{4}\supset \GR{B}{3}\supset \GR{G}{2}$, &   
$\omin(\eus D_4)\leadsto \co'=\co(3,1^4) \leadsto \co''=\mathsf{G}_2(a_1)$; \\
$\GR{B}{4}\supset \GR{D}{4}\supset (\GR{B}{3}, \vp_3)$, & 
$\omin(\eus B_4)\leadsto \co'=\co(3,1^5) \leadsto  \co''=\co(3,2,2)$. \\
\end{tabular}
\end{center}
The first two triples occur in \cite{bk94} (Corollary~6.2 and Theorem~6.6), whereas the last one, which involves item~7 of Table~\ref{table:odin}, is new.


\begin{thebibliography}{P95}

\bibitem{AP}   {\sc R.S.~Avdeev} and {\sc A.V.~Petukhov}.
Spherical actions on flag varieties, 
{\it Matem. Sb.}, {\bf 205}\,(2014),  no.\,9, 3--48 (Russian). English  translation:
{\it Sb. Math.},  {\bf 205}\,(2014),  no.\,9, 1223--1263.  

\bibitem{schichten}    {\sc W.~Borho.}
\"Uber Schichten halbeinfacher Lie-Algebren, 
{\it Invent. Math.}, {\bf 65}\,(1981), 283--317.

\bibitem{bokr}    {\sc W.~Borho} and {\sc H.~Kraft}.
\"Uber Bahnen und deren Deformationen bei linearen Aktionen reduktiver Gruppen,
{\it Comment. Math. Helv.},  {\bf 54}\,(1979), no.\,1, 61--104.

\bibitem{brion}  {\sc M.~Brion}. Invariants et covariants des groupes 
alg\'ebriques r\'eductifs, In: ``{\it Th\'eorie des invariants et g\'eom\'etrie
des vari\'et\'es quotients}'' (Travaux en cours, t.\,{\bf 61}), 83--168,
Paris: Hermann, 2000.

\bibitem{bk94} {\sc R.K.\,Brylinski} and {\sc B.\,Kostant}.
Nilpotent orbits, normality, and Hamiltonian group actions,
{\it J. Am. Math. Soc.}, {\bf 7}\,(1994), no.\,2, 269--298.

\bibitem{CM} {\sc D.H.\,Collingwood} and {\sc W.\,McGovern}.
{\it ``Nilpotent orbits in semisimple Lie algebras''},  New York: Van Nostrand Reinhold, 1993.

\bibitem{Dy52} {\sc E.B.\,Dynkin}.
Semisimple subalgebras of semisimple Lie algebras,
{\it Matem. Sb.}, {\bf 30}\,(1952), no.\,2, 349--462 (Russian).
English translation: {\it  Am. Math. Soc. Transl.}, 
II~Ser., {\bf 6}\,(1957), 111--244.


\bibitem{fogarty}  
{\sc J.\,Fogarty}. Fixed point schemes,  
{\it Am. J. Math.},  {\bf 95}\,(1973), 35--51.

\bibitem{t41}
{\sc V.V.\,Gorbatsevich, A.L.\,Onishchik} and {\sc E.B.\,Vinberg}.
``{\it Lie Groups and Lie Algebras}" III
(Encyclopaedia Math. Sci., vol.~{\bf 41}) Berlin: Springer 1994.

\bibitem{he79}
{\sc W.~Hesselink}.  The normality of closures of orbits in a Lie algebra. 
{\it Comment. Math. Helv.} {\bf 54}\,(1979), no.\,1, 105--110.


\bibitem{kim79} {\sc B.N.\,Kimelfeld}.
Reductive groups which act locally transitively on flag manifolds of orthogonal groups,
{\it Tr. Tbilis. Mat. Inst. Razmadze}, {\bf 62}\,(1979), 49--75 (Russian). 
English  translation: {\it Sel. Math. Sov.} {\bf 4}\,(1985), 107--130.

\bibitem{kim} {\sc B.N.\,Kimelfeld}.
Homogeneous domains on flag manifolds,
{\it J. Math. Anal. Appl.} {\bf 121}\,(1987), 506--588.

\bibitem{kn90} {\sc F.\,Knop}.
Weylgruppe und Momentabbildung, {\it Invent. Math.}, {\bf 99}\,(1990), 1--23.

\bibitem{KP82}    {\sc H.\,Kraft} and {\sc C.\,Procesi}. 
On the geometry of conjugacy classes in classical groups,
{\it Comment. Math. Helv.}, {\bf 57}\,(1982), 539--602.

\bibitem{kr79} {\sc M.\,Kr\"amer}.
Sph\"arische Untergruppen in kompakten zusammenh\"angenden Liegruppen. 
{\it Compos. Math.}, {\bf 38}\,(1979), 129--153.


\bibitem{ls88}
{\sc T.\,Levasseur} and {\sc S.P.\, Smith}.
Primitive ideals and nilpotent orbits in type $G_2$,  
{\it J. Algebra}, {\bf 114}\,(1988), no.~1, 81--105.

\bibitem{L72} {\sc D.~Luna}.
Sur les orbites ferm{\'e}es des groups alg{\'e}briques
r{\'e}ductifs, {\it Invent. Math.}, {\bf 16}\,(1972), 1--5.

\bibitem{L07} {\sc D.~Luna}.
La vari\'et\'e magnifique mod\`ele,
{\it J. Algebra},  {\bf 313}\,(2007), 292--319.

\bibitem{on62}    {\sc A.L.\,Onishchik}.
Inclusion relations between transitive compact transformation groups,
{\it Tr. Mosk. Mat. O-va}, {\bf 11}\,(1962), 199--242 (Russian). English translation:
 {\it  Am. Math. Soc. Transl.},  II~Ser., {\bf 55}\,(1966), 3--58.

\bibitem{on69}     {\sc A.L.\,Onishchik}. 
Decompositions of reductive Lie groups,
{\it Matem. Sb.},  {\bf 80}\,(1969), 553--599 (Russian). English translation:
{\it Math. USSR--Sbornik}, {\bf 9}\,(1969), 515--554.

\bibitem{diss}        {\sc D.\,Panyushev}. 
Complexity and rank of actions in invariant theory, 
{\it J. Math. Sci.} (New York),  {\bf 95}\,(1999), 1925--1985.

\bibitem{p99}      {\sc D.\,Panyushev}.
On spherical nilpotent orbits and beyond,
{\it Ann. Inst. Fourier} {\bf 49}\,(1999), 1453--1476.

\bibitem{p01}      {\sc D.\,Panyushev}. 
The exterior algebra and `Spin' of an orthogonal $\g$-module,
{\it Transformation Groups},  {\bf 6}\,(2001), 371--396.

\bibitem{p03}       {\sc D.\,Panyushev}.  
Some amazing properties of spherical nilpotent orbits,  
{\it Math. Z.}, {\bf 245}\,(2003), 557--580.  

\bibitem{omin} {\sc D.\,Panyushev}.
Projections of the minimal nilpotent orbit in a simple Lie algebra and secant varieties,
{\it Math. Proc. Camb. Phil. Soc.}, {\bf 175}, no.\,3 (2023), 595--624.

\bibitem{p24}  {\sc D.\,Panyushev}.
Orbits and invariants for coisotropy representations,
(\href{http://arxiv.org/abs/2405.01897}{\tt \bfseries arXiv}, 2024), 23 pp.

\bibitem{vp72} 
{\sc E.B.~Vinberg} and {\sc V.L.~Popov}. On a class of quasihomogeneous
affine varieties, {\it Math. USSR-Izv.}, {\bf 6}\,(1972), 743--758.

\end{thebibliography}
\end{document}